\documentclass[11pt]{article}

\usepackage{amsfonts}
\usepackage{amsmath,amssymb}
\usepackage{amscd}
\usepackage{amsthm}
\usepackage{fancyhdr}
\usepackage{color}
\usepackage{hyperref} 
\usepackage{tikz-cd}
\usepackage{txfonts}
\usepackage{centernot}
\usepackage{mathtools}
\usepackage{ stmaryrd }

\usepackage{tensor} 
\usepackage{cancel} 

\theoremstyle{plain}
\newtheorem{lem}{Lemma}[section] 
\newtheorem{thm}[lem]{Theorem}

\newtheorem{cor}[lem]{Corollary}
\newtheorem{prop}[lem]{Proposition}

\theoremstyle{definition}

\newtheorem{defnx}[lem]{Definition}
\newenvironment{defn}
  {\pushQED{\qed}\defnx}
  {\popQED\enddefnx}

  \newtheorem{definex}[lem]{Definition-Example}
\newenvironment{defex}
  {\pushQED{\qed}\definex}
  {\popQED\enddefnx}

\newtheorem{examplex}[lem]{Example}
\newenvironment{ex}
  {\pushQED{\qed}\examplex}
  {\popQED\endexamplex}

  \theoremstyle{remark}
\numberwithin{equation}{lem}

\newtheorem{obsx}[lem]{Remark}
\newenvironment{rk}
  {\pushQED{\qed}\obsx}
  {\popQED\endobsx}

\newcommand{\KR}{\mathbb{R}}
\newcommand{\RR}{\mathbb{R}} 

\newcommand{\CG}{\mathcal{G}}
\newcommand{\cG}{\mathcal{G}} 
\newcommand{\CH}{\mathcal{H}}
\newcommand{\cH}{\mathcal{H}} 

\newcommand{\CR}{\mathcal{R}}

\newcommand{\CU}{\mathcal{U}}

\newcommand{\g}{\mathfrak{g}}

\newcommand{\st}{\hspace{.05in}:\hspace{.05in}}
\newcommand{\y}{\hspace{.09in}\text{and}\hspace{.09in}}

\newcommand{\fto}{\rightarrow}

\newcommand{\soutar}{\rightrightarrows}
\newcommand{\tto}{\rightrightarrows} 

\newcommand{\gdt}{t}
\newcommand{\gds}{s}
\newcommand{\gdm}{m}
\newcommand{\gdu}{u}
\newcommand{\gdi}{\tau} 

\newcommand{\bt}{\mathbf{t}}
\newcommand{\bs}{\mathbf{s}}
\newcommand{\bm}{\mathbf{m}}
\newcommand{\bu}{\mathbf{u}}
\newcommand{\bi}{\boldsymbol{\tau}}

\newcommand{\tit}{\widetilde{t}}
\newcommand{\tis}{\widetilde{s}}
\newcommand{\tim}{\widetilde{m}}
\newcommand{\tiu}{\widetilde{u}}
\newcommand{\tii}{\widetilde{\tau}}

\newcommand{\tiR}{\widetilde{R}}

\newcommand{\Ima}{\mathrm{Im}}
\newcommand{\Fr}{\mathrm{Fr}}
\newcommand{\GL}{\mathrm{GL}}
\newcommand{\pr}{\mathrm{pr}}
\newcommand{\rank}{\mathrm{rank}}
\newcommand{\id}{\mathrm{id}}
\newcommand{\Hom}{\mathrm{Hom}}
\newcommand{\Diff}{\mathrm{Diff}}

\newcommand{\frs}{$s$-bisection}
\newcommand{\frt}{$t$-bisection}
\newcommand{\sbis}{\mathrm{sbis}}
\newcommand{\tbis}{\mathrm{tbis}}

\newcommand{\dphi}{d_{\phi}}
\newcommand{\dphig}{d_{\phi_{g}}}

\usepackage[backend=bibtex,style=numeric-comp, doi=false,isbn=false,url=false,eprint=false, maxnames=4]{biblatex} 
\title{PB-Groupoids vs VB-Groupoids}

\author{Francesco Cattafi \thanks{Julius-Maximilians-Universit\"at W\"urzburg, Germany. 
		Email: \texttt{francesco.cattafi@uni-wuerzburg.de}		\newline		
		The author was partially supported by the DFG under Walter Benjamin project 460397678 (Germany), and is a member of the GNSAGA - INdAM (Italy).
		} \footnotemark[3]
		\and
Alfonso Garmendia
	\thanks{Max Planck Institute for Mathematics, Bonn, Germany.
		Email: \texttt{garmendia@mpim-bonn.mpg.de}\newline
		Part of the work was done in  Centre de Recerca Matemàtica and Universitat Politècnica de Catalunya, Spain. The author was supported by the Spanish State Research Agency MCIU/AEI / 10.13039/501100011033, through the Severo Ochoa and María de Maeztu Program for Centers and Units of Excellence in R\&D (project CEX2020-001084-M) and the grant PID2019-103849GB-I00.} 
		\thanks{Both authors were also partially supported by the AGAUR project 2021 SGR 00603 Geometry of Manifolds and Applications, GEOMVAP with PI Prof.\ Eva Miranda.} 
}

\begin{document}

\date{\today}

\maketitle

\begin{abstract}
In this paper we establish the principal bundle counterpart of the well-known and widely applied notion of vector bundle groupoid (VB-groupoid). In particular, we provide a general notion of principal bundle groupoid (PB-groupoid) as a diagram of Lie groupoids and principal bundles, together with the action of a (strict) Lie 2-groupoid. This recovers as particular cases various constructions involving structural Lie 2-groups, Lie groupoids and Lie groups. 

Moreover, given any VB-groupoid, we detect which frames are compatible with the underlying structure and show that the sets of such frames has a natural PB-groupoid structure, with the action of the general linear 2-groupoid of the appropriate ranks. We use this construction, as well as that of the associated bundle induced by a 2-representation, to extend the classical correspondence between vector bundles and principal bundles to a new correspondence between VB-groupoids and PB-groupoids. We conclude with several examples
and applications.
\end{abstract}

\begin{center}
\end{center}



\begin{center}
\end{center}

\tableofcontents

\section{Introduction}



In differential geometry, the notions of principal bundles and vector bundles over a smooth manifold $M$ are intimately related by the standard correspondence 
\[
\left\{
\begin{array}{cc}
\text{Principal bundles over $M$} \\
\text{with structural Lie group $\GL(k,\RR)$}
\end{array}
 \right\}
\longleftrightarrow
 \left\{
\begin{array}{cc}
 \text{Vector bundles over $M$}\\
 \text{of rank $k$} 
\end{array}
\right\}.
\]

The overall goal of this paper is to (make sense and) generalise the correspondence above to ``principal bundle objects'' and ``vector bundle objects'' over a Lie groupoid $\cG \tto M$, namely
\[
\left\{
\begin{array}{cc}
\text{PB-groupoids over $\cG \tto M$} \\
\text{with structural Lie 2-groupoid $\GL(l,k)$}
\end{array}
 \right\}
\longleftrightarrow
 \left\{
\begin{array}{cc}
 \text{VB-groupoids over $\cG \tto M$}\\
 \text{of rank $(l,k)$} 
\end{array}
\right\}.
\]


VB-groupoids are already used in many geometric contexts. For instance, they appear when integrating fibrewise linear Poisson structures on vector bundles \cite{CosteDazordWeinstein87}; they are the models of two-term representations up to homotopy \cite{GraciaMehta17}, and in particular of the adjoint representation of a Lie groupoid \cite{AriasAbadCrainic13}; their cohomology control the deformation theory of Lie groupoids \cite{CrainicMestreStruchiner20}; they can be used for blow-up constructions in relation with index theory \cite{DebordSkandalis21}. Accordingly, just as the frames of vector bundles encode them into principal bundles, and therefore constitute a fundamental tool in their study, our notion of PB-groupoids and the correspondance above will shed more light on VB-groupoids and their applications. See also the end of this introduction for further motivations.

Let us first quickly recall the main objects involved in the classical correspondence. Given a manifold $M$ of dimension $n$, a frame at $x \in M$ is a linear isomorphism $\RR^n \to T_x M$. The set of frames of $M$ at all points, denoted by $\Fr(M)$ or $\Fr(TM)$, is a smooth manifold of dimension $n + n^2$, and is moreover a principal $\GL(n,\RR)$-bundle over $M$, with the right action given by the composition between frames and matrices. More generally, given any vector bundle $E \to M$ of rank $k$, its {\bf frame bundle} is the set
\[
\Fr(E) := \{ \RR^k \to E_x \text{ linear isomorphism}, x \in M \}.
\]
It is a smooth manifold of dimension $n + k^2$ and a principal $\GL(k,\RR)$-bundle over $M$. The case $E = TM$ (where $k = n$) recovers $\Fr(M)$. In other words, one has the construction 
\[
\Bigg\{ \text{ Vector bundles of rank $k$ } \Bigg\} \longrightarrow \Bigg\{ \text{ Principal $\GL(k,\RR)$-bundles } \Bigg\}
\]
\[
  (E \to M) \mapsto (\Fr(E) \to M).
\]

On the other hand, given a principal $H$-bundle $P \to M$ and a representation of $H$ on a $\KR$-vector space $V$, the {\bf associated bundle} is
\[
 P[V] := (P \times V)/H.
\]
If $V$ has dimension $k$, then $P[V]$ is a vector bundle of rank $k$. Then one has the construction
\[
\Bigg\{ \text{ Principal $H$-bundles and $H$-representations of dimension $k$ } \Bigg\} \longrightarrow \Bigg\{ \text{ Vector bundles of rank $k$ } \Bigg\}
\]
\[
(P \to M, V) \mapsto (P[V] \to M).
\]


Taking $P = \Fr(E)$ and $V = \RR^k$, for $k$ the rank of $E \to M$, one has an isomorphism of vector bundles $P[V] \cong E$. Conversely, given a principal $\GL(k,\RR)$-bundle $P \to M$, one gets an isomorphism of principal bundles $\Fr(P[\RR^k]) \cong P$ using the canonical $\GL(k,\RR)$-representation on $\RR^k$.

\

In order to extend the PB-VB correspondence to the world of Lie groupoids, one needs first of all to understand the objects on both sides. A {\bf VB-groupoid} is a vector bundle object in the category of Lie groupoid, i.e.\ a diagram
\[\begin{tikzcd}
	{E_\cG} & \cG \\
	{E_M} & M
	\arrow[from=1-1, to=2-1, shift left=.5ex]
	\arrow[from=1-1, to=2-1, shift right=.5ex]
	\arrow[from=1-2, to=2-2, shift left=.5ex]
	\arrow[from=1-2, to=2-2, shift right=.5ex]
	\arrow[from=1-1, to=1-2]
	\arrow[from=2-1, to=2-2]
\end{tikzcd}\]
consisting of two vector bundles and two Lie groupoids, such that the various morphisms are compatible with both structures (see Definition \ref{def_VB_groupoid} later on). This notion was introduced in 1988 by Pradines and it is now a standard concept in the theory of Lie groupoids; see e.g.\   \cite{GraciaMehta17,BCdH16, Mackenzie05} for an overview.

On the other hand, the notion of {\bf PB-groupoid} is not yet well established and the first goal of our paper is to fill 
this gap. If one tries to mimick the definition of VB-groupoid, i.e.\ a principal bundle object in the category of Lie groupoids, one encounters a diagram of the kind
\[\begin{tikzcd}
{P_\cG} & \cG \\
{P_M} & M
\arrow[from=1-1, to=2-1, shift left=.5ex]
\arrow[from=1-1, to=2-1, shift right=.5ex]
\arrow[from=1-2, to=2-2, shift left=.5ex]
\arrow[from=1-2, to=2-2, shift right=.5ex]
\arrow[from=1-1, to=1-2]
\arrow[from=2-1, to=2-2]
\end{tikzcd}\]
with a principal action of a group object in the category of Lie groupoids. More precisely, one uses a Lie 2-group $H_2\soutar H_1$, with $H_2$ acting on $P_\CG$ and $H_1$ acting on $P_M$, as in the definition of principal bundle groupoids given in \cite{GarmendiaPaycha23} by Garmendia and Paycha. However,  as we will see, this notion is very restrictive when working with the frames of a VB-groupoid. Furthermore,  besides standard principal bundles with a structure Lie group, one might be interested in considering principal bundles with structural Lie groupoids.



Indeed, in the literature there are various notions for ``higher'' principal bundles, including:
\begin{itemize}
	\item \textbf{Principal bundles over groupoids} as in \cite[Definition 2.33]{LTX}, \textbf{PBG-groupoids} as in \cite[Definition 2.1]{Mackenzie87paper} and \textbf{$G$-groupoids} as in \cite[Definition 4.4]{BruceGrabowskaGrabowski17}, where a Lie group acts on a Lie groupoid;
\item  \textbf{Generalized principal bundles} as in \cite[Theorem 2.1]{CastrillonRodriguez23}, where a Lie group bundle acts on a fibre bundle;	
	\item \textbf{Principal groupoid bundles} as in \cite[Section 5.7]{MoerdijkMrcun03}, \cite[Definition 2.1]{R}, \cite[Definition 3.1]{I} and \cite[Section 2.2]{H}, where a Lie groupoid acts on a smooth map;
	\item \textbf{Principal 2-bundles} as in \cite[Definition 6.1.5]{NW}, where a Lie 2-group acts on a fibred pair groupoid; 
	\item \textbf{Principal bundle groupoids} as in \cite{GarmendiaPaycha23} and \textbf{principal 2-bundles over Lie groupoids} as in \cite{CCK22} and \cite{HerreraOrtiz23}, where a Lie 2-group acts on a Lie groupoid.
\end{itemize}

In this article we extend the definition of \textbf{PB-groupoid} to a diagram as above, with a (strict) Lie \textit{2-groupoid} acting free and properly on the Lie groupoid $P_\cG \tto P_M$. This notion not only generalises all the definitions above, but also fits as the natural structure carried by the ``adapted frames'' of a VB-groupoid.


Indeed, while the construction of the associated vector bundle generalises easily to the PB-groupoid case, that of the frame bundle does not: taking all the possible frames of the vector bundle $E_\cG \to \cG$ does not yield a groupoid structure over the frame bundle of $E_M \to M$. Accordingly, one has to restrict to a class of frames of $E_\cG$ which are ``adapted'' to the VB-groupoid structure (which we will call \textit{$s$-bisection frames}), obtaining a diagram of the form
\[\begin{tikzcd}
{\Fr^\sbis(E_\cG)} & \cG \\
{\Fr(C) \times_M \Fr(E_M)} & M
\arrow[from=1-1, to=2-1, shift left=.5ex]
\arrow[from=1-1, to=2-1, shift right=.5ex]
\arrow[from=1-2, to=2-2, shift left=.5ex]
\arrow[from=1-2, to=2-2, shift right=.5ex]
\arrow[from=1-1, to=1-2]
\arrow[from=2-1, to=2-2]
\end{tikzcd}\]
Moreover, $\Fr^\sbis(E_\cG)$ admits a natural (principal) action of the Lie 2-groupoid $\GL(l,k)$, where $l$ is the rank of the core $C \to M$ of $E_\cG$ and $k$ the rank of the base $E_M \to M$, and defines therefore a PB-groupoid. Here $\GL(l,k)$ is a direct generalisation of the group of invertible matrices of a given dimension, which we describe explicitly but can also be seen as a particular case of the general linear 2-groupoid from \cite{DelHoyoStefani19} and of the 2-gauge groupoid from \cite{HerreraOrtiz23}.

\paragraph{Main results and structure of the paper}

Section 2 contains a solid background on VB-groupoids and related objects, including (strict) Lie 2-groupoids and (two versions of) 2-vector bundles, in order to make this article as self-contained as possible. Indeed, these last two topics are well-known to the experts but we could find no detailed/complete sources in the literaure. The most important object of this section, which will play a key role in the next ones, is the {\bf general linear 2-groupoid} $\GL(l,k)$ (Definition \ref{ex_GL_2_groupoid}).



Our contributions start in section 3. First of all we introduce two notions of {\bf 2-representation of Lie 2-groupoids} (Definition \ref{def_representation_2_groupoid}), one of which, to the best of our knowledge, has not appear previously in the existing literature. Then we provide our first main definition: the {\bf general notion of PB-groupoid} over a Lie groupoid (Definition \ref{def_PB_groupoid}). Finally, we describe the {\bf associated bundle of a PB-groupoid} and a 2-representation, and we prove that it is a VB-groupoid (Theorem \ref{thm_from_PB_to_VB}).

Section 4 is devoted to the opposite construction, namely the {\bf ``adapted frame bundle'' $\Fr^\sbis(E_\cG)$ of a VB-groupoid $E_\cG \tto E_M$} (Definition \ref{def_special_frames}). The main result here, i.e.\ the fact that $\Fr^\sbis(E_\cG) \tto \Fr(C) \times_M \Fr(E_M)$ is a PB-groupoid (Theorem \ref{thm_from_VB_to_PB}), is a consequence of the facts that it is a set-theoretical groupoid (Theorem \ref{thm_special_frames_are_groupoid}), it is endowed with a natural $\GL(l,k)$-action (Theorem \ref{prop_2-action_on_frames}), and it is smooth (Theorem \ref{thm_special_frames_are_principal_bundles}). We conclude with the proof that the constructions in section 3 and 4 are inverse to each other, and therefore yield a {\bf 1-1 correspondence between VB-groupoids and PB-groupoids} (Theorem \ref{thm_1-1_correspondence_VB_PB}).

Last, in Section \ref{sec_examples} we illustrate our main results by {\bf several interesting examples}, e.g.\ applying the frame bundle construction to the VB-groupoids with trivial core or with trivial base, as well as to various relevant cases of tangent Lie groupoids. Moreover, we discuss the {\bf compatibility of our construction with duality}: the PB-groupoid associated to the dual $E_\cG^*$ of a VB-groupoid $E_\cG$ turns out to be canonically isomorphic to the PB-groupoid associated to $E_\cG$ (Corollary \ref{cor_duality_PB}).

\paragraph{Further motivations and future applications}

On the one hand, this paper belongs to the recent programme of investigating classical objects of differential geometry compatible with a Lie groupoid. Such objects are often known as {\bf multiplicative structures} and have naturally emerged in various contexts in Poisson geometry \cites{BursztynDrummond19, Kosmann16}. In turn, this programme is strictly related to (and motivated by) the description of such geometric objects on {\bf manifolds with singularities}, e.g.\ leaf spaces of foliations or orbit spaces of group actions. Indeed, these spaces are not smooth manifolds but can be formalised in the framework of differentiable stacks, i.e.\ Morita equivalence classes of Lie groupoids \cite{DelHoyo13}. In particular, the frame bundle of the tangent VB-groupoid $T\cG \tto TM$ (Example \ref{ex_PB_of_tangent_groupoid}) represents the first step 
in order to obtain a notion of {\bf frame bundle of a differentiable stack}. More generally, one could investigate principal bundles over a differentiable stack, complementing what done in \cite{DelHoyoOrtiz20} for VB-stacks, which led already to further insights and applications to
representations up to homotopy.

On the other hand, an important application of the classical PB-VB correspondence is the possibility to encode geometric structures on the tangent bundle of a manifold (e.g.\ Riemannian structures, symplectic structures, complex structures, etc.) via reductions of the structure group of its frame bundle, giving rise to the well-known theory of $G$-structures \cite{Sternberg64, Kobayashi95}. In several geometric problems where they arise, Lie groupoids are naturally equipped with compatible geometric structures. As a future application of the results in our paper, we aim at establishing a theory of {\bf multiplicative $G$-structure-like objects on a Lie groupoid $\cG$} in terms of reductions of the structural Lie 2-groupoid of its adapted frame bundle. Two crucial tasks in this problem will involve the generalisation of the classical tautological/soldering form $\theta \in \Omega^1 (\Fr(M),\RR^n)$ to the frame bundle of $T\cG$, as well as the appropriate notion of principal connection for PB-groupoids. This last object would naturally pave the way to the development of a {\bf Chern-Weil theory for Lie groupoids}, extending the results in \cite{LaurentGengouxTuXu07}, \cite{BiswasChatterjeeKoushikNeumann24} and \cite{HerreraOrtiz23}.

Last, one of the main principle of Lie theory is the strict relation between global and infinitesimal objects, which gives the possibility to transform non-linear problems into linear ones. In particular, it is known that multiplicative structures on Lie groupoids induce (interesting, and possibly easier) {\bf infinitesimally multiplicative structures} on the corresponding Lie algebroids. Accordingly, we plan to adapt our construction of the adapted frame bundle to a VB-algebroid, obtaining what one should call {\bf PB-algebroid}, which we expect to be the infinitesimal counterpart of our PB-groupoids.

\paragraph{Notations and conventions}

All manifolds and maps are smooth, unless explicitly stated otherwise. Principal group(oid) actions are always considered from the right, while representations are considered from the left. The symbol $\GL(k)$ denotes the group of $k \times k$ invertible \textit{real} matrices. 

In many instances, we will consider several different groupoid structures, possibly on the same space. For all of them, we use the letters $\gds$ and $\gdt$ to denote the source and target, the letter $\gdm$ and the sign $\circ$ or $\circ_\gdm$ to denote the multiplication, the letter $\gdu$ for the unit and $\gdu_x$ to denote the unit at a point $x$ in the base, and the letter $\gdi$ or the symbol $(-)^\gdi$ for the inverse and $g^\gdi$ to denote the inverse of an arrow $g$. We use tilde for the structure maps of VB-groupoids ($\tis$, $\tit$, $\tim$, etc.) and boldface for the structure maps of PB-groupoids ($\bs$, $\bt$, $\bm$, etc.).

\paragraph{Acknowledgements} 


We would like to thank Chenchang Zhu, Luca Vitagliano, Alfonso Tortorella, Madeleine Jotz and Ivan Struchiner for their interest in our work and for useful discussions and suggestions at various stages of this project, as well as Miquel Cueca Ten for suggesting several missing references.

We are also grateful to the anonymous reviewers for their comments, which significantly improved this paper; in particular, for pointing out the relations with the fat groupoid, which led to a more conceptual understanding and a better proof of Theorem \ref{thm_special_frames_are_groupoid}.

\section{Background on 2-objects}

Through this paper we will use, without recalling the basics, the theory of Lie groupoids, Lie algebroids, and principal bundles (with structural Lie groups or Lie groupoids). For an introduction on these topics we refer to \cite{Mackenzie05, CrainicFernandes11, MoerdijkMrcun03}.

\subsection{VB-groupoids}

In this section we review some basics on VB-groupoids; for proofs and further details we refer to \cite{GraciaMehta17,BCdH16, Mackenzie05}. The reader familiar with this topic can safely skip this section but should take note of our emphasis on a non-standard terminology, which will be relevant in the later sections: the {\bf rank} $(l,k)$ of a VB-groupoid consists of the ranks of the core and of the base vector bundles.


\begin{defn}\label{def_VB_groupoid}
A {\bf VB-groupoid of rank $(l,k)$} is a commutative diagram
\[\begin{tikzcd}
	{E_\cG} & \cG \\
	{E_M} & M
	\arrow[from=1-1, to=2-1, shift left=.5ex, "\tis"]
	\arrow[from=1-1, to=2-1, shift right=.5ex, "\tit"']
	\arrow[from=1-2, to=2-2, shift left=.5ex, "s"]
	\arrow[from=1-2, to=2-2, shift right=.5ex, "t"']
	\arrow["{\pi_\cG}", from=1-1, to=1-2]
	\arrow["{\pi_M}", from=2-1, to=2-2]
\end{tikzcd}\]
such that
\begin{itemize}
 \item $\pi_{\cG} \colon E_{\cG} \to \cG$ is a vector bundle of rank $l+k$ and $\pi_M: E_M \to M$ is a vector bundle of rank $k$;
 \item $\cG \tto M$ and $E_{\cG} \tto E_M$ are Lie groupoids;
 \item all the groupoid structure maps are morphisms of vector bundles. \qedhere
\end{itemize}
\end{defn}

Equivalent definitions of a VB-groupoid can be found in \cite[Section 3.1]{GraciaMehta17}. For instance, the third bullet can be replaced by the following three conditions:
\begin{itemize}
	\item $(\tis, s)$ and $(\tit, t)$ are morphisms of vector bundles;
	\item $(\pi_{\cG}, \pi_M)$ is a morphism of Lie groupoids;
	\item the interchange law hold, i.e.\ 
	\[
	\tim ( e_1 + e_3, e_2 + e_4 ) = \tim (e_1 ,e_2) + \tim (e_3, e_4)  
	\]
	for every $(e_1,e_2), (e_3,e_4) \in E_{\cG} \tensor[_{s}]{\times}{_{t}} E_{\cG}$ such that $\pi_\cG (e_1) = \pi_\cG (e_3)$ and $\pi_\cG (e_2) = \pi_\cG (e_4)$.
\end{itemize}

\begin{defn}\label{def_core}
The (right-){\bf core} of a VB-groupoid is
\[
 C := \ker\left(\,\tis\,\right)|_M = \left\{\, e \in E_{\cG} \st x \in M\,,\,\, \pi_{\cG} (e) = \gdu (x)\,,\,\, \tis (e) = 0 \,\right\} \subseteq E_{\cG}.
\]
The {\bf core-anchor} $\rho$ is the restriction of $\tit \colon E_{\cG} \to E_M$ to the core, i.e.\ $\rho:= \tit|_C: C \to E_M$.
\end{defn}

The core has a natural structure of vector bundle of rank $l$ over $M$, so that the core-anchor is a morphism of vector bundles over $\id_M$. The {\it left-core} is defined analogously using $\tit$ instead of $\tis$, and it is canonically isomorphic to $C$.

\begin{ex}\label{ex_tangent_VB_groupoid}
Given any Lie groupoid $\cG \tto M$, the tangent bundles $T \cG \to \cG$ and $TM \to M$, together with the differential of the structure maps of $\cG$, define a VB-groupoid of rank $(\dim(\cG) - \dim(M), \dim(M))$, called the {\bf tangent VB-groupoid}:
\[\begin{tikzcd}
	{T\cG} & \cG \\
	{TM} & M
	\arrow[from=1-1, to=2-1, shift left=.5ex, "ds"]
	\arrow[from=1-1, to=2-1, shift right=.5ex, "dt"']
	\arrow[from=1-2, to=2-2, shift left=.5ex, "s"]
	\arrow[from=1-2, to=2-2, shift right=.5ex, "t"']
	\arrow["{\pi_\cG}", from=1-1, to=1-2]
	\arrow["{\pi_M}", from=2-1, to=2-2]
\end{tikzcd}\]
Its core coincides with the Lie algebroid $A \to M$ of $\cG$ and its core-anchor with the anchor $\rho := dt|_A \colon A \to TM$ of $A$.
\end{ex}

\begin{ex}\label{ex_cotangent_VB_groupoid}
Given any Lie groupoid $\cG \tto M$ with algebroid $A \to M$, the cotangent bundle $T^*\cG \to \cG$ and the dual bundle $A^* \to M$ define a VB-groupoid of rank $(\dim(M), \dim(\cG) - \dim(M))$, called the {\bf cotangent VB-groupoid}.
\[\begin{tikzcd}
	{T^*\cG} & \cG \\
	{A^*} & M
	\arrow[from=1-1, to=2-1, shift left=.5ex, "\tis"]
	\arrow[from=1-1, to=2-1, shift right=.5ex, "\tit"']
	\arrow[from=1-2, to=2-2, shift left=.5ex, "s"]
	\arrow[from=1-2, to=2-2, shift right=.5ex, "t"']
	\arrow[from=1-1, to=1-2]
	\arrow[from=2-1, to=2-2]
\end{tikzcd}\]
The source and target maps of $T^*\cG \tto A^*$  are defined, for $g\in \CG$, $\xi_g\in T_g^* \CG$ and $v_{(-)}\in A_{(-)}$, as follows:
\[
 \tis (\xi_g) (v_{s(g)}):= - \xi_g (d_e L_g (d_{1_{s(g)}}\gdi(v))), \quad \quad \tit (\xi_g) (v_{t(g)}) := \xi_g (d_e R_g (v) ).
\]
The multiplication between two elements $\xi_1 \in T^*_{g_1}\cG$ and $\xi_2 \in T^*_{g_2}\cG$ such that $\tis(\xi_1) = \tit (\xi_2)$ is determined, for any composable $v_1 \in T_{g_1}\cG$ and $v_2 \in T_{g_2}\cG$, by
\[
 (\xi_1 \circ \xi_2) (dm (v_1,v_2) ) := \xi_1 (v_1) + \xi_2 (v_2).
\]
The core of $T^*\cG$ coincides with the cotangent bundle $T^* M \to M$ and its core-anchor with the dual map $\rho^*: T^*M \to A^*$ of the anchor $\rho: A \to TM$.
\end{ex}

\begin{defex}\label{ex_dual_VB_groupoid}
Given any VB-groupoid $E_\cG \tto E_M$ of rank $(l,k)$ with core $C \to M$, the dual bundles $E_\cG^* \tto \cG$ and $C^* \to M$ define a VB-groupoid of rank $(k,l)$, called the {\bf dual VB-groupoid}:
\[\begin{tikzcd}
	{E_\cG^*} & \cG \\
	{C^*} & M
	\arrow[from=1-1, to=2-1, shift left=.5ex, "\tis"]
	\arrow[from=1-1, to=2-1, shift right=.5ex, "\tit"']
	\arrow[from=1-2, to=2-2, shift left=.5ex, "s"]
	\arrow[from=1-2, to=2-2, shift right=.5ex, "t"']
	\arrow[from=1-1, to=1-2]
	\arrow[from=2-1, to=2-2]
\end{tikzcd}\]
The source and target maps of $E_\cG^* \tto C^*$ are defined, for $g\in \CG$, $\xi_g\in (E^*_\CG)_g$ and $v_{(-)}\in C_{(-)}$, as follows:
\[
 \tis (\xi_g) (v_{s(g)}):= - \xi_g (\tim(0_g, \tii(v) ) ), \quad \quad \tit (\xi_g) (v_{t(g)}) := \xi_g (\tim (v, 0_g) ).
\]
The multiplication between two elements $\xi_1 \in (E_{\cG}^*)_{g_1}$ and $\xi_2 \in (E_{\cG}^*)_{g_2}$ such that $\tis(\xi_1) = \tit (\xi_2)$ is determined, for any composable $v_1 \in (E_{\cG})_{g_1}$ and $v_2 \in (E_{\cG})_{g_2}$, by
\[
 (\xi_1 \circ \xi_2) (\tim (v_1,v_2) ) := \xi_1 (v_1) + \xi_2 (v_2).
\]
The core of $E_\cG^*$ coincides with the dual vector bundle $E_M^* \to M$ and its core-anchor with the dual map $E_M^* \to C^*$ of the core-anchor $C \to E_M$ of $E_\cG$. In particular, the cotangent VB-groupoid of a Lie groupoid (Example \ref{ex_cotangent_VB_groupoid}) is the dual VB-groupoid of its tangent VB-groupoid (Example \ref{ex_tangent_VB_groupoid}).
\end{defex}

\begin{ex}
A {\bf double vector bundle} (DVB) is a commutative diagram of vector bundles
\[\begin{tikzcd}
	E & B \\
	A & M
	\arrow["{\pi^E_A}", from=1-1, to=2-1]
	\arrow["{\pi^B}", from=1-2, to=2-2]
	\arrow["{\pi^E_B}", from=1-1, to=1-2]
	\arrow["{\pi^A}", from=2-1, to=2-2]
\end{tikzcd}\]
where all the structure maps (i.e.\ projections, additions, multiplications by scalars and zero sections) are vector bundle morphisms; see e.g.\ \cite[Chapter 9]{Mackenzie05} for further details. Any DVB is a VB-groupoid with rank $(\rank(\pi^E_B) - \rank (\pi^A), \rank (\pi^A))$, since any vector bundle is a bundle of (abelian) Lie groups, i.e.\ it is a special Lie groupoid with $s = t$. Moreover, its core coincides with the DVB core, i.e.\ the intersection $C = (\pi^A)^{-1} (0_M) \cap (\pi^B)^{-1} (0_M)$. 
\end{ex}

\begin{ex}[trivial core]\label{ex_VB_groupoid_trivial_core}
Given a Lie groupoid $\cG \tto M$ and a representation $E_M \to M$ of $\cG$, the action groupoid $\cG \times_M E_M \tto E_M$ defines a VB-groupoid of rank $(0,k)$, together with the induced vector bundle $\cG \times_M E_M \to \cG$:
\[\begin{tikzcd}
	{\cG \times_M E_M} & \cG \\
	{E_M} & M
	\arrow[from=1-1, to=2-1, shift left=.5ex, "\tis"]
	\arrow[from=1-1, to=2-1, shift right=.5ex, "\tit"']
	\arrow[from=1-2, to=2-2, shift left=.5ex, "s"]
	\arrow[from=1-2, to=2-2, shift right=.5ex, "t"']
	\arrow["{\pr_1}", from=1-1, to=1-2]
	\arrow[from=2-1, to=2-2]
\end{tikzcd}\]
By construction, the core of $\cG \times_M E_M$ is the zero-vector bundle. Conversely, any VB-groupoid $E_\cG \tto E_M$ with core $C = 0$ is isomorphic to the action groupoid $\cG \times_M E_M \tto E_M$, for the representation of $\cG$ on $E_M$ given by $g \cdot e := \tit_g (v)$, with $v \in (E_\cG)_g$ uniquely defined by $\tis_g (v) = e$.

Note that, if a tangent VB-groupoid $T\cG \tto TM$ (Example \ref{ex_tangent_VB_groupoid}) has rank $(0,k)$, i.e.\ its Lie algebroid $A \to M$ is the zero vector bundle, then $\cG \tto M$ must be an \'etale Lie groupoid. In this case, $T\cG \cong \cG \times_M TM$ for the canonical representation of any \'etale Lie groupoid $\cG \tto M$ on $TM$.
\end{ex}

The dual of a VB-groupoid with trivial core does not have trivial core but has trivial base. This means that such dual is not the VB-groupoid associated to the dual representation $E^*_M$. On the other hand, VB-groupoids with trivial base are encoded into representations of their core, as discussed in Example \ref{ex_VB_groupoid_trivial_base} below. Applying the general construction from Example \ref{ex_dual_VB_groupoid} to Example \ref{ex_VB_groupoid_trivial_core} one obtains precisely the VB-groupoid with trivial base associated to the dual representation $E^*_M$.

\begin{ex}[trivial base]\label{ex_VB_groupoid_trivial_base}
Given a Lie groupoid $\cG \tto M$ and a rank $l$ representation $C \to M$ of $\cG$, one has the following Lie groupoid structure on $\cG \times_M C \tto M \times \{0\}$:
\[
 \tis (g,c) = s(g), \quad \quad \tit (g,c) = t(g),
\]
\[
 \tim ((g_1,c_1), (g_2,c_2)) = (g_1 g_2, c_1 + g_2 \cdot c_2).
\]
This defines a VB-groupoid of rank $(l,0)$, together with the induced vector bundle $\cG \times_M C \to \cG$:
\[\begin{tikzcd}
	{\cG \times_M C} & \cG \\
	{M \times \{0\}} & M
	\arrow[from=1-1, to=2-1, shift left=.5ex, "\tis"]
	\arrow[from=1-1, to=2-1, shift right=.5ex, "\tit"']
	\arrow[from=1-2, to=2-2, shift left=.5ex, "s"]
	\arrow[from=1-2, to=2-2, shift right=.5ex, "t"']
	\arrow["{\pr_1}", from=1-1, to=1-2]
	\arrow[from=2-1, to=2-2]
\end{tikzcd}\]
Conversely, any VB-groupoid $E_\cG \tto E_M = M \times \{0\}$ with core $C$ is isomorphic to the Lie groupoid $\cG \times_M C \tto M \times 0$ discussed above, for the representation of $\cG$ on $C$ given by $g \cdot c = \tim (0_g, c)$.

Note that, if a tangent VB-groupoid $T\cG \tto TM$ (Example \ref{ex_tangent_VB_groupoid}) has rank $(l,0)$, i.e.\ $TM \to M$ is the zero vector bundle, then $M$ is $0$-dimensional. If we assume that $M$ is connected, then $G = \cG \tto \{*\}$ must be a Lie group and the core $C$ its Lie algebra $\g$. In this case, $TG \cong G \times \g$ for the adjoint representation of any Lie group $G$ on $\g$.
\end{ex}

\begin{ex}[pullback]\label{ex_pullback_VB_groupoid}
Given a vector bundle $\pi: E_M \to M$ of rank $k$ and a Lie groupoid $\cG \tto M$, the pullback groupoid of $\cG$ w.r.t.\ $\pi$ defines a VB-groupoid of rank $(k,k)$:
\[\begin{tikzcd}
	{E_M \tensor[_{\pi}]{\times}{_{t}} \cG \tensor[_{s}]{\times}{_{\pi}} E_M} & \cG \\
	{E_M} & M
	\arrow[from=1-1, to=2-1, shift left=.5ex, "\tis"]
	\arrow[from=1-1, to=2-1, shift right=.5ex, "\tit"']
	\arrow[from=1-2, to=2-2, shift left=.5ex, "s"]
	\arrow[from=1-2, to=2-2, shift right=.5ex, "t"']
	\arrow["{\pr_2}", from=1-1, to=1-2]
	\arrow["{\pi}", from=2-1, to=2-2]
\end{tikzcd}\]
The core of $E_M \tensor[_{\pi}]{\times}{_{t}} \cG \tensor[_{s}]{\times}{_{\pi}} E_M$ coincides with $E_M$ and its core-anchor is the identity $E_M \to E_M$.
\end{ex}

The following simple example will be crucial for carrying out our theory; further insights will be provided later.

\begin{defex}\label{ex_trivial_VB_groupoid}
 Any pair $(l,k)$ of natural numbers defines a VB-groupoid of rank $(l,k)$, which we call the {\bf canonical VB-groupoid} $\KR^{(l,k)}$. We will use here the notation $\GL(l,k)_0:=\Hom(\KR^l,\KR^k)$, the reasons of which will be clear in Definition \ref{def:trivial.gl.l.k}. The groupoid $\KR^{(l,k)}$ is given by
\[\begin{tikzcd}
\KR^{(l,k)}_2:=	\RR^l \times \RR^k \times \GL(l,k)_0 & \GL(l,k)_0 \\
\KR^{(l,k)}_1:=	\RR^k \times \GL(l,k)_0 & \GL(l,k)_0
	\arrow[from=1-1, to=2-1, shift left=.5ex, "\tis"]
	\arrow[from=1-1, to=2-1, shift right=.5ex, "\tit"']
	\arrow[from=1-2, to=2-2, shift left=.5ex]
	\arrow[from=1-2, to=2-2, shift right=.5ex]
	\arrow["{\pr}", from=1-1, to=1-2]
	\arrow["{\pr}", from=2-1, to=2-2]
\end{tikzcd}\]
Here the rows are the trivial vector bundles of ranks $l+k$ and $k$ over the manifold $\GL(l,k)_0$. Moreover, the groupoid on the left column is the action groupoid of the following (abelian) action of $\KR^l$ on $\KR^k\times \GL(l,k)_0$:
$$\KR^l\times\KR^k\times\GL(l,k)_0 \to \KR^k\times \GL(l,k)_0, \quad (w,v,d) \mapsto (d(w)+v,d).$$
Finally, the groupoid $\GL(l,k)_0 \tto\GL(l,k)_0$ 
on the right column is identified with the unit groupoid of $\GL(l,k)_0$. A VB-groupoid with this last property is actually equivalent to an anchored 2-vector bundle (see Lemma \ref{lemma_2_vector_bundles_2_VB_groupoids} below).

The core of $\RR^{(l,k)}$ is the trivial vector bundle of rank $l$ over $\GL(l,k)_0$, and the core-anchor is
\[
\RR^l \times \GL(l,k)_0 \to \RR^k \times \GL(l,k)_0, \quad (w,d) \mapsto (d(w),d). \qedhere
\]
\end{defex}

\subsection{2-vector bundles}

In the literature there are various ``higher'' generalisations of vector bundles, often induced from abstract categorical points of view. With our specific applications in mind, here we consider two notions which we will use in the rest of the paper, and which might have appeared elsewhere under different names.

\begin{defn}\label{def_2_vector_bundle}
A {\bf 2-graded vector bundle} is a pair $(E_1,E_0)$ of vector bundles over the same manifold $M$. An {\bf anchored 2-graded vector bundle} (or, shortly, an {\bf anchored 2-vector bundle}) is a 
2-graded vector bundle $(E_1,E_0)$ endowed with a vector bundle map $\delta \colon E_1 \to E_0$ over the identity $\id_M$, which we call {\bf anchor}. This is represented by the diagram
\[\begin{tikzcd}
	{E_1} & M \\
	{E_0} & M
	\arrow[from=1-1, to=1-2]
	\arrow["{\id_M}", from=1-2, to=2-2]
	\arrow["\delta"', from=1-1, to=2-1]
	\arrow[from=2-1, to=2-2]
\end{tikzcd}\]
and denoted, more compactly, by $E_1\xrightarrow{\delta} E_0\fto M$.
\end{defn}

Anchored 2-vector bundles generalise the standard notion of \textit{anchored vector bundle}, i.e.\ a vector bundle $E_1 \to M$ together with a vector bundle map $E_1 \to TM$; in our case, $TM$ is replaced by an arbitrary vector bundle $E_0 \to M$. From a different point of view, an anchored 2-vector bundle is simply a chain complex $(E_i,\delta_i)$ of vector bundles over the same base, with $E_i = 0$ for every $i \neq 0,1$.

Another reason for the adjective ``anchored'' comes from the following lemma, which is part of the folklore knowledge in the field; we are adding it here for the sake of self-containment. Note that this lemma can also be seen as a special instance of the Dold-Kan correspondence between chain complexes of vector bundles and simplicial vector bundles; in this case, the chain complexes have only two non-zero terms and the corresponding simplicial vector bundles are the nerves of VB-groupoids over the units.

\begin{lem}[anchored 2-vector bundles and VB-groupoids over the unit groupoid]\label{lemma_2_vector_bundles_2_VB_groupoids}
Any anchored 2-vector bundle $E_1 \xrightarrow{\delta} E_0 \to M$ induces a VB-groupoid over the unit groupoid $M \tto M$:
\[\begin{tikzcd}
	{E_1  \times_M E_0} & M \\
	{E_0} & M
	\arrow[from=1-1, to=2-1, shift left=.5ex, "\tis"]
	\arrow[from=1-1, to=2-1, shift right=.5ex, "\tit"']
	\arrow[from=1-2, to=2-2, shift left=.5ex]
	\arrow[from=1-2, to=2-2, shift right=.5ex]
	\arrow[from=1-1, to=1-2]
	\arrow[from=2-1, to=2-2]
\end{tikzcd}\]
where $E_1 \times_M E_0$ is the action groupoid of the following (abelian) groupoid action of $E_1\fto M$ on $E_0\fto M$:
$$ E_1 \times_M E_0 \to E_0, \quad (v_1,v_0) \mapsto \delta (v_1)+v_0.$$
Conversely, given a VB-groupoid $E_\CG\soutar E_M$ over the unit groupoid $\CG = M \tto M$, then $E_1:=C$ and $E_0:=E_M$ define an anchored 2-vector bundle, with anchor map equal to the core-anchor $\delta:=\tit|_{E_1} \colon E_1\fto E_0$ (Definition \ref{def_core}).
\end{lem}

\begin{rk}[2-graded vector bundles and VB-groupoids over the unit groupoid]
Note that a 2-graded vector bundle would not be enough to define the VB-groupoid structure described in Lemma \ref{lemma_2_vector_bundles_2_VB_groupoids}, since one needs the anchor map $\delta$ in order to define the groupoid structure on $E_1 \times_M E_0 \tto E_0$. Conversely, any VB-groupoid $E_\cG \tto E_M$ over the unit groupoid $M \tto M$ induces a 2-vector bundle $(E_1 := C,E_0 := E_M)$, where the anchor $\delta \colon E_1 \to E_0$ is the core-anchor $C \to E_M$.
\end{rk}

\begin{ex}[2-graded vector spaces]\label{ex_2-vector_spaces}
When $M = \{*\}$, a 2-anchored vector space $V_1 \xrightarrow{\delta} V_0$ is simply a 2-term chain complex of vector spaces, hence it is equivalent to the notion of 2-vector spaces introduced in \cite[Definition 3.1, Proposition 3.4]{BaezCrans04}, i.e.\ the structure underlying a strict Lie 2-algebra.

On the other hand, a 2-graded vector bundle boils down to a 2-graded vector space $(V_1,V_0)$. In particular, any pair $(l,k)$ of natural numbers defines the 2-graded vector bundle $(\KR^l,\KR^k)$ over a point.

Notice that there is no canonical anchor $\delta \colon \RR^l \to \RR^k$ making $(\KR^l,\KR^k)$ into an anchored 2-vector space.  Each choice of such an anchor would yield, via Lemma \ref{lemma_2_vector_bundles_2_VB_groupoids}, a VB-groupoid $\RR^l \times \RR^k \tto \RR^k$ over the trivial unit groupoid $\{*\} \tto \{*\}$, with core-anchor equal to $\delta$.
\end{ex}


\begin{defex}\label{ex_trivial_VB_groupoid_as_2-vector_bundle}
Any pair $(l,k)$ of natural numbers defines an anchored 2-vector bundle $E^{(l,k)}$ over $M = \Hom (\RR^l,\RR^k)$, which we call the {\bf canonical 2-anchored vector bundle}, given by the trivial vector bundles
\[ 
E^{(l,k)}_1 = \RR^l \times \Hom(\RR^l, \RR^k) \to \Hom(\RR^l, \RR^k), \quad \quad  E^{(l,k)}_0 = \RR^k \times \Hom(\RR^l, \RR^k) \to \Hom(\RR^l, \RR^k) 
\]
 and the anchor map
\[
\delta\colon E^{(l,k)}_1 \to E^{(l,k)}_0, \quad (v_1,d) \mapsto (d(v_1),d).
\]
Applying the correspondence from Lemma \ref{lemma_2_vector_bundles_2_VB_groupoids}), the canonical 2-anchored vector bundle $E^{(l,k)}$ can be equivalently described precisely by the canonical VB-groupoid $\RR^{(l,k)}$ from Example \ref{ex_trivial_VB_groupoid}.

More generally, any 2-graded vector bundle $(E_1,E_0)$ over $M$ defines an anchored 2-vector bundle $(E_1 \times_M \Hom(E_1,E_0), E_0 \times_M \Hom(E_1,E_0))$ with anchor given by
\[
\delta (v_1, d) = (d(v_1), d).
\]
Notice that, for $M$ a point and $(E_1,E_0) = (\RR^l,\RR^k)$, we recover precisely the canonical 2-anchored vector bundle $E^{(l,k)}$ described above.
\end{defex}

\begin{ex}\label{ex_vector_bundes_as_2VB}
Ordinary vector bundles can be viewed as anchored 2-vector bundles in three non-equivalent ways:
\begin{itemize}
 \item when $E_0 = M$ (i.e.\ $E_0$ has rank $0$) and $\delta$ coincides with the projection $E_1 \to M$;
 \item when $E_1 = M$ (i.e.\ $E_1$ has rank $0$) and $\delta \colon M \to E_0$ is the zero section;
 \item when $E_0 = E_1$ and $\delta = \id$. \qedhere
\end{itemize}
\end{ex}

\subsection{Lie 2-groupoids}

As for 2-vector bundles, Lie 2-groupoids have several inequivalent definitions in the literature, as discussed in Remark \ref{rk_strict_vs_weak} below. Here we follow e.g.\ \cite{MehtaTang11}, \cite{DelHoyoStefani19} and \cite{BrahicOrtiz19} and pick the notion which will appear naturally when looking at the frame bundle of a VB-groupoid.

In general, since we will consider different groupoid structures on the same space, we will from now on adopt the notation $\gds_{ij}, \gdt_{ij},\gdm_{ij},\gdu_{ij},\gdi_{ij}$ for the structure maps of a Lie groupoid $\CH_i\soutar\CH_j$; for $\CH_i\soutar M$ we will use $\gds_{M}, \gdt_{M},\gdm_{M},\gdu_{M},\gdi_{M}$.

\begin{defn}\label{def_Lie_2_groupoid}\label{rk_Lie_2_groupoids_double_groupoids}
	A {\bf double Lie groupoid}  (see e.g.\ \cite{BrownMackenzie92} or \cite[Definition 3.1 and 3.2]{MehtaTang11}) is a commutative diagram of Lie groupoids
	\[\begin{tikzcd}
		{\CH_2} & \CH_1 \\
		{\CH_0} & M
		\arrow[from=1-1, to=2-1, shift left=.5ex]
		\arrow[from=1-1, to=2-1, shift right=.5ex]
		\arrow[from=1-2, to=2-2, shift left=.5ex]
		\arrow[from=1-2, to=2-2, shift right=.5ex]
		\arrow[from=1-1, to=1-2, shift left=.5ex]
		\arrow[from=1-1, to=1-2, shift right=.5ex]
		\arrow[from=2-1, to=2-2, shift left=.5ex]
		\arrow[from=2-1, to=2-2, shift right=.5ex]
	\end{tikzcd}\]
	where the following three conditions are satisfied:
	\begin{enumerate}
		\item  all the source and targets maps are Lie groupoid morphisms;
		\item  the interchange law
		$$(g_1 \circ_{m_{20}} g_2)\circ_{m_{21}}(g_3 \circ_{m_{20}} g_4)=(g_1 \circ_{m_{21}} g_3)\circ_{m_{20}}(g_2 \circ_{m_{21}} g_4)$$
		holds for all $g_i\in \CH_2$ such that the compositions above make sense;
		\item the map $(\gds_{21},\gds_{20})\colon\CH_2\fto \CH_1 \tensor[_{\gds_{M}}]{\times}{_{\gds_{M}}} \CH_0$ is a surjective submersion.
	\end{enumerate}

A {\bf Lie 2-groupoid} $\CH_2 \tto \CH_1 \tto \CH_0$ (see e.g.\ \cite[Example 3.7]{MehtaTang11}) is a double Lie groupoid where the base groupoid $\CH_0 \tto M=\CH_0$ is the unit groupoid. In other words it is a commutative diagram of Lie groupoids
 \[\begin{tikzcd}
	{\CH_2} && {\CH_1} \\
	& {\CH_0}
	\arrow[from=1-1, to=1-3, shift left=.5ex]
	\arrow[from=1-1, to=1-3, shift right=.5ex]
	\arrow[from=1-3, to=2-2, shift left=.5ex]
	\arrow[from=1-3, to=2-2, shift right=.5ex]
	\arrow[from=1-1, to=2-2, shift left=.5ex]
	\arrow[from=1-1, to=2-2, shift right=.5ex]
\end{tikzcd}\]
such that conditions 1 and 2 above hold (condition 3 holds automatically, since $\CH_1 \tensor[_{\gds_{M}}]{\times}{_{\gds_{M}}} \CH_0\cong \CH_1$ and the map $(\gds_{21},\gds_{20})$ reduces to the surjective submersion $\gds_{21}$). See also \cite[Section 4]{DelHoyoStefani19} and \cite[Section 4.1]{BrahicOrtiz19} for an equivalent formulation in terms of 2-categories.
\end{defn}

\begin{rk}[strict vs weak Lie 2-groupoids]\label{rk_strict_vs_weak}
The object we described above is also known as {\it strict} Lie 2-groupoid, in order to distinguish it from  {\it weak} Lie 2-groupoids, defined as simplicial manifolds satisfying a horn-filling condition; see e.g.\ \cite{Zhu09} for the precise definition and further references,
and \cite[Sections 2-4]{DelHoyoStefani19} for a comparison between the two notions. In this paper we will never use the simplicial definition, and by ``Lie 2-groupoid'' we will always refer to the stricter notion from \ref{def_Lie_2_groupoid}. The same holds for Lie-2-groups, recalled in Definition \ref{def:lie2grp} below.
\end{rk}

\begin{rk}[anchored 2-vector bundles vs Lie 2-groupoids]
 While every vector bundle $\pi: E \to M$ is automatically a Lie groupoid with $s = t = \pi$, an anchored 2-vector bundle $E_1 \xrightarrow{\delta} E_0 \to M$ (Definition \ref{def_2_vector_bundle}) is not \textit{automatically} a Lie 2-groupoid (Definition \ref{def_Lie_2_groupoid}). Indeed, the vector bundles $E_1 \to M$ and $E_0 \to M$ are Lie groupoids (being bundles of abelian groups), but the vector bundle map $\delta \colon E_1 \to E_0$ is simply a Lie groupoid morphism, i.e.\ it is not the same thing as a Lie groupoid structure, since we lack a splitting of $\delta$ to define a groupoid unit. On the other hand, by Lemma \ref{lemma_2_vector_bundles_2_VB_groupoids} any anchored 2-vector bundle $E_1 \xrightarrow{\delta} E_0 \to M$ \textit{induces canonically} a VB-groupoid $E_1\times_M E_0\soutar E_0$ over the unit groupoid, which can be therefore seen, in view of Definition \ref{rk_Lie_2_groupoids_double_groupoids}, as a special double Lie groupoid over the unit groupoid, i.e.\ a special Lie 2-groupoid.
\end{rk}

\begin{ex}[Lie groupoids]\label{ex_Lie_groupoids_are_Lie_2_groupoids}
 Ordinary Lie (1-)groupoids can be viewed as Lie 2-groupoids when $\CH_1 = \CH_2$.
\end{ex}

\begin{ex}[Lie groups]\label{ex_Lie_groups_are_Lie_2_groupoids}
 Ordinary Lie (1-)groups can be viewed as Lie 2-groupoids when $\CH_1 = \CH_2$ and $\CH_0 = \{*\}$. This is of course a special instance (for $\CH_0 = \{*\}$) of Example \ref{ex_Lie_groupoids_are_Lie_2_groupoids}.
\end{ex}

Notice that, given a Lie 2-groupoid $\cH_2 \tto \cH_1 \tto \cH_0$, the condition $\CH_0 = \CH_1$ is not sufficient to obtain a Lie (1-)groupoid: indeed, the interchange law would not be satisfied unless $\cH_2$ is abelian. Furthemore, $\cH_2$ would be forced to be a bundle of Lie groups: indeed, $t_{21}$ is a groupoid morphism over $\id_{\cH_0}$, hence $s_{20} = s_{10} \circ t_{21}$. But in the case $\cH_1 = \cH_0$ we have $s_{10} = t_{10}$, therefore $s_{20} = t_{10} \circ t_{21} = t_{20}$. For the same reason, the condition $\CH_0 = \CH_1 = \{*\}$ is not sufficient to obtain a Lie (1-)group.

\begin{defn}\label{def:lie2grp}
	A \textbf{Lie $2$-group} is a Lie $2$-groupoid where $\CH_0$ is a singleton; see  e.g.\ \cite[section 3.1]{GarmendiaZambon21}
	for a more explicit definition.

	Given any Lie $2$-groupoid $\CH_2\soutar\CH_1\soutar\CH_0$ and an element $d\in\CH_0$, then the pair $\left((\CH_2)_d, (\CH_1)_d\right)$ of isotropy groups forms a Lie $2$-group, called the \textbf{isotropy $2$-group} at $d$.
	\end{defn}

Notice that the isotropy 2-group at $d$ can be defined replacing $(\cH_2)_d$ by the set of elements of $\cH_2$ starting and ending at points in $(\cH_1)_d$. This is a consequence of the definition of Lie 2-groupoid, more precisely, the facts that $t_{20} = t_{10} \circ t_{21} = t_{10} \circ s_{21}$ and $s_{20} = s_{10} \circ s_{21} = s_{10} \circ t_{21}$.


\begin{ex}[crossed modules]
 A {\bf crossed module of Lie groups} is a pair of Lie groups $(H,G)$ together with a map $d\colon H\fto G$ and an action $C\colon G\fto {\rm Aut}(H)$, satisfying, for any $g\in G$ and $h,h'\in H$,
 $$d (C_g h) =gd(h) g^{-1} \y C_{d(h)}h'=hh'h^{-1}. $$ 

 It is well known in the literature that Lie 2-groups and crossed modules of Lie groups are in 1-1 correspondence; in the set-theoretical setting, the first published proof seems to be \cite[Theorem 1]{BrownSpencer76}. Let us remind the reader of this correspondence, which makes crossed modules a special case of Lie 2-groupoids.

 Given a crossed module $(H,G,d,C)$, one defines a Lie 2-groupoid $H\times G\soutar G\soutar \{*\}$, where the groupoid structure on $H\times G\soutar G$ is the action groupoid of the canonical action of $H$ in $G$ given by $(h,g)\mapsto d(h) g$, and the groupoid structure of $H\times G\soutar \{*\}$ is given by the semidirect product $H\rtimes G$ induced by $C$.
 
 On the other hand, given a Lie 2-group $\CH_2\soutar \CH_1\soutar \{*\}$, the pair $(H=\ker(\gds_{21}),G=\CH_1)$ is a crossed module, together with $d: =\gdt_{21}|_H\colon H\fto G$ and the conjugate action $C$ of the Lie subgroup $\gdu(G)\subset (\CH_2\soutar \{*\})$  on $H\subset \CH_2$, i.e.\ for all $g\in G$ and $h\in H$,
 \[
C_g h := (\gdu_{21}(g))\,\circ_{20} h \circ_{20} \,(\gdu_{21}(g))^{\tau_{20}}. \qedhere  
 \]
\end{ex}

The following example of Lie 2-groupoid will be fundamental in the rest of the paper.

\begin{defex}\label{ex_GL_2_groupoid}\label{def:trivial.gl.l.k} Any pair $(l,k)$ of natural numbers defines a Lie 2-groupoid $\GL(l,k)$, which we call the \textbf{general linear 2-groupoid of rank $(l,k)$}, given by
\[\begin{tikzcd}
\GL(l,k)_2 := \left\{ \left( d, 
\left(
\begin{matrix}
	A & JB \\
	0 & B
\end{matrix}
\right)
 \right) \in \Hom(\KR^l,\KR^k)\times \GL(l+k) \st (I_l+Jd)\in \GL(l) \y (I_k+dJ)\in \GL(k) \right\}
\\
  \GL(l,k)_1 := \Hom (\RR^l,\RR^k) \times \GL(l) \times \GL(k), 
\\
\GL(l,k)_0 := \Hom (\RR^l, \RR^k).
\arrow[from=1-1, to=2-1, shift left=.9ex]
\arrow[from=1-1, to=2-1, shift right=.9ex]
\arrow[from=2-1, to=3-1, shift left=.9ex]
\arrow[from=2-1, to=3-1, shift right=.9ex]
\end{tikzcd}\] 

\begin{itemize}
	\item The groupoid structure on $\GL(l,k)_2 \soutar \GL(l,k)_0$ is the unique one with
	 source and target maps $\gdt_{20},\gds_{20}\colon \GL(l,k)_2 \fto \GL(l,k)_0$ given by:
		\[\begin{tikzcd}
		d & \left( d, \left(
		\begin{matrix}
		A & JB \\
		0 & B
		\end{matrix}
		\right)\right) & ((I+dJ)B)^{-1}dA
		\arrow[from=1-2, to=1-1, maps to, swap, "\gdt_{20}"]
		\arrow[from=1-2, to=1-3, maps to, "\gds_{20}"]
		\end{tikzcd},\]
	 and multiplication map $\gdm_{20} \colon \GL(l,k)_2 \tensor[_{\gds_{20}}]{\times}{_{\gdt_{20}}} \GL(l,k)_2 \to \GL(l,k)_2$ given by the matrix multiplication:
		\[\begin{tikzcd}
		\left(\,\,\left( d, 
		\left(
		\begin{matrix}
		A & JB \\
		0 & B
		\end{matrix}
		\right)
		\right)\,\, ,\,\, \left( \gdt_{20}\left( d, 
		\left(
		\begin{matrix}
		A & JB \\
		0 & B
		\end{matrix}
		\right)\right)\,,\, 
		\left(
		\begin{matrix}
		A' & J'B' \\
		0 & B'
		\end{matrix}
		\right)
		\right)\,\,\right) 
		& \left( d, 
		\left(
		\begin{matrix}
		AA' & (AJ'B^{-1}+J)BB' \\
		0 & BB'
		\end{matrix}
		\right)
		\right).
		\arrow[from=1-1, to=1-2, mapsto, "\gdm_{20}"]
		\end{tikzcd}\]

\item The groupoid structure on $\GL(l,k)_2 \tto \GL(l,k)_1$ is the unique one with
source and target maps $\gdt_{21},\gds_{21}\colon \GL(l,k)_2 \fto \GL(l,k)_1$ given by:
	\[\begin{tikzcd}
	(d\,,\, A\,,\, (I+dJ)B) & \left( d \,,\, 
	\left(
	\begin{matrix}
	A & JB \\
	0 & B
	\end{matrix}
	\right)
	\right) & (d\,,\, (I+Jd)^{-1}A\,,\, B)
	\arrow[from=1-2, to=1-1, maps to, swap, "\gdt_{21}"]
	\arrow[from=1-2, to=1-3, maps to, "\gds_{21}"]
	\end{tikzcd},\]
and multiplication map $\gdm_{21} \colon \GL(l,k)_2 \tensor[_{\gds_{21}}]{\times}{_{\gdt_{21}}} \GL(l,k)_2 \to \GL(l,k)_2$ given by:
	\[\begin{tikzcd}
	\left(\,\,\left( \,d, 
	\left(
	\begin{matrix}
	A & J(I+dJ')B' \\
	0 & (I+dJ')B'
	\end{matrix}
	\right)\,
	\right) \,,\, \left( \,d, 
	\left(
	\begin{matrix}
	(I+Jd)^{-1}A & J'B' \\
	0 & B'
	\end{matrix}
	\right)\,
	\right)\,\, \right)
	& \left( d, 
	\left(
	\begin{matrix}
	A & (JdJ'+J+J')B' \\
	0 & B'
	\end{matrix}
	\right)
	\right) 
		\arrow[from=1-1, to=1-2, mapsto, "\gdm_{21}"]
	\end{tikzcd}.\]
	\item The groupoid $\GL(l,k)_1 \tto \GL(l,k)_0$ is given by the action groupoid of the canonical right action of $\GL(l)\times\GL(k)$ on $\GL(l,k)_0$, given by $d \cdot (A,B) := B^{-1} \circ d \circ A$, i.e.
	\[\begin{tikzcd}
		d  & (d,A,B)\ar[l,maps to,swap,  "\gdt_{10}"]  \ar[r, maps to, "\gds_{10}"] & B^{-1} \circ d \circ A
	\end{tikzcd}.\]
\end{itemize}

Moreover, notice that the isotropy Lie $2$-group $((\GL(l,k)_2)_d, (\GL(l,k)_1)_d) $ of $\GL(l,k)$ at any element $d\in \GL(l,k)_0$, as in Definition \ref{def:lie2grp}, recovers precisely the data of the crossed module $(K_1,K_0)$ associated to $d$ defined in \cite[Section 3.1]{ShengZhu12}, which was introduced to integrate Lie 2-algebras.
\end{defex}

\begin{defex}\label{ex_general_linear_2_groupoid}
Given any 2-graded vector bundle $E = (E_1,E_0)$ over $X$, its {\bf general linear 2-groupoid} $\GL(E)$ is the Lie 2-groupoid defined as follows:
\begin{itemize}
	\item $\GL(E)_0=\Hom(E_1,E_0)$;
	\item an element of $\GL(E)_1$ with source $d_x\colon (E_1)_x\fto (E_0)_x$ and target $d_y\colon (E_1)_y\fto (E_0)_y$ is a couple of linear isomorphisms $A\colon (E_1)_x\fto (E_1)_y$ and $B\colon  (E_0)_x\fto (E_0)_y$ such that $d_y\circ A= B\circ d_x$; 
	\item an element of $\GL(E)_2$ with source $(d_x,d_y,A,B)$ and target  $(d_x,d_y,A',B')$ is given by a map $J\colon (E_0)_x\fto (E_1)_y$ such that $J\circ d_x= A-A'$ and $d_y \circ J=B-B'$.
\end{itemize}
The structure maps of $\GL(E)$ are analogous to those of $\GL(l,k)$ from Definition \ref{ex_GL_2_groupoid}; for more details, see \cite[Section 5]{DelHoyoStefani19} and \cite[Section 4.3]{BrahicOrtiz19}, where this concept was first introduced (see also Remark \ref{rk_general_linear_2_groupoid} below). More precisely, in \cite[Remark 5.6]{DelHoyoStefani19} the Lie 2-groupoid $\GL(E)$ was denoted by $\GL^\prime(E)$, in order to distinguish it from a similar notion which allows more general arrows, while in \cite{BrahicOrtiz19} it was called 2-gauge groupoid and denoted by 2-$\mathrm{Gau}(E)$. 
\end{defex}

\begin{rk}[General linear 2-groupoid and representations up to homotopy]\label{rk_general_linear_2_groupoid}

The Lie 2-groupoid $\GL(E)$ was introduced in \cite{DelHoyoStefani19}  in order to interpret 2-term representations up to homotopy (RUTHs) of a Lie groupoid $\cG$ on a 2-graded vector bundle $E$ as morphisms $\cG \to \GL(E)$, in analogy to what happens for ordinary representations of Lie groupoids on a vector bundle. Moreover, $\GL(E)$ is actually equivalent to the gauge 2-groupoid 2-$\mathrm{Gau}(E)$ from \cite[Section 4.3]{BrahicOrtiz19}, which was used in order to integrate a 2-term RUTH of a Lie algebroid to a 2-term RUTH of a Lie groupoid.
The general idea of ``representing'' RUTHs was first announced in \cite{Mehta14}.
\end{rk}

\begin{ex}
	Applying the construction of Example \ref{ex_general_linear_2_groupoid} to the canonical 2-graded vector space $(\KR^l,\KR^k)$ over $X=\{*\}$ yields precisely the general Lie 2-groupoid $\GL(l,k)$ from Example \ref{ex_GL_2_groupoid}.

	On the other hand, the general linear 2-groupoid of the 2-graded vector bundle underlying the canonical anchored 2-vector bundle $E^{(l,k)}$ from Example \ref{ex_trivial_VB_groupoid_as_2-vector_bundle} is intuitively $\GL(l,k)$ ``enlarged'' by multiplying by a suitable amount of copies of $\Hom (\RR^l,\RR^k)$; more precisely,
	\begin{itemize}
		\item $\GL(E^{(l,k)})_0= \GL(l,k)_0 \times \Hom (\RR^l,\RR^k)$;
		\item $\GL(E^{(l,k)})_i\cong \GL(l,k)_i \times \Hom (\RR^l,\RR^k)\times  \Hom (\RR^l,\RR^k)$ for $i\in\{1,2\}$.
	\end{itemize}
	
	It is also interesting to consider the case of a 2-graded vector bundle $E$ where either $E_0 \to X$ or $E_1 \fto X$ has rank $0$ (Example \ref{ex_vector_bundes_as_2VB}); then $E$ is an ordinary vector bundle and its general linear 2-groupoid recovers the general linear (1-)groupoid $\GL(E) \tto X$, called sometimes the gauge or Atiyah groupoid of $E$.
\end{ex}

\section{From PB-groupoids to VB-groupoids}

\subsection{2-Actions}

In order to define a PB-groupoid, we discuss first actions of Lie 2-groupoids on Lie groupoids along pairs of smooth maps, which generalise ordinary actions of Lie groupoids on manifolds along smooth maps (see e.g.\ \cite[Section 5.3]{MoerdijkMrcun03}). For other particular cases of 2-actions present in the literature, we refer to the works cited in section \ref{sec_PB_groupoids}.

\begin{defn}\label{def_action_Lie_2_groupoid}
A {\bf 2-action} of a Lie 2-groupoid $\CH_2 \tto \CH_1 \tto \CH_0$ (Definition \ref{def_Lie_2_groupoid}) on a Lie groupoid $P_2 \tto P_1$ consists of
\begin{itemize}
 \item two smooth maps $\mu_2 \colon P_2 \to \CH_0$ and $\mu_1 \colon P_1 \to \CH_0$;
 \item a (right) action of $\CH_2 \tto \CH_0$ on $P_2$ along $\mu_2$;
 \item a (right) action of $\CH_1 \tto \CH_0$ on $P_1$ along $\mu_1$;
 \end{itemize}
 such that
 
 \begin{enumerate}
 	\item for any $p \in P_2$ one has $\mu_2(p)=\mu_1(\bs(p))=\mu_1(\bt(p))$;
 	\item the action maps define a Lie groupoid morphism
 	\[\begin{tikzcd}
 		{P_2 \tensor[_{\mu_2}]{\times}{_{t_{20}}} \CH_2} & P_2 \\
 		{P_1 \tensor[_{\mu_1}]{\times}{_{t_{10}}} \CH_1} & P_1
 		\arrow[from=1-1, to=2-1, shift left=.5ex, "{(\bs, s_{21})}"]
 		\arrow[from=1-1, to=2-1, shift right=.5ex, "{(\bt, t_{21})}"']
 		\arrow[from=1-2, to=2-2, shift left=.5ex, "\bs"]
 		\arrow[from=1-2, to=2-2, shift right=.5ex, "\bt"']
 		\arrow[from=1-1, to=1-2, "\Phi"]
 		\arrow[from=2-1, to=2-2, "\varphi"]
 	\end{tikzcd}\]
  where on the left column the Lie groupoid structure is given by restricting the product groupoid ${P_2 \times \CH_2}\soutar P_1\times \CH_1$. \qedhere
 \end{enumerate}  
\end{defn}

Note that the submanifold ${P_2 \tensor[_{\mu_2}]{\times}{_{t_{20}}} \CH_2} \subseteq P_2 \times \CH_2$ defines a subgroupoid of the product groupoid ${P_2 \times \CH_2}\soutar P_1\times \CH_1$ precisely because $\mu_2$ and $\mu_1$ satisfy condition 1. Moreover, the action groupoid $P_2 \tensor[_{\mu_2}]{\times}{_{t_{20}}} \CH_2\soutar P_2$ defines a double Lie groupoid over the action groupoid $P_1 \tensor[_{\mu_1}]{\times}{_{t_{10}}} \CH_1\soutar P_1$ precisely because the action maps $(\Phi,\phi)$ form a Lie groupoid morphism.

\begin{ex}[actions of Lie 2-groupoids on themselves]\label{ex_action_Lie_2_groupoid_on_itself}
Any Lie 2-groupoid $\CH_2 \tto \CH_1 \tto \CH_0$ acts on $\CH_2 \tto \CH_1$, say from the right, with $\mu_2 = s_{20}$ and $\mu_1 = s_{10}$. Indeed, the action is given by the standard Lie groupoid actions of $\CH_2 \tto \CH_0$ on $\CH_2$ along $\mu_2$ and of $\CH_1 \tto \CH_0$ on $\CH_1$ along $\mu_1$. Condition 1 of Definition \ref{def_action_Lie_2_groupoid} is automatically fulfilled by the fact that a Lie 2-groupoid is a commutative diagram, while condition 2 follows from the interchange law.
\end{ex}

\begin{ex}[actions of Lie groupoids]\label{ex_Lie_2_action_by_Lie_groupoid_case_2}
 If $\CH_1 = \CH_2$ (Example \ref{ex_Lie_groupoids_are_Lie_2_groupoids}), then condition 1 in Definition \ref{def_action_Lie_2_groupoid} becomes trivial while condition 2 implies the following special property: every global bisection $b$ of $\CH_2 \tto \CH_0$ defines a Lie groupoid automorphism of $P_2 \tto P_1$, given by
 \[
 P_2 \to P_2, \quad p_2 \mapsto p_2 \cdot b (\mu_2 (p_2)),
 \]
 \[
  P_1 \to P_1, \quad p_1 \mapsto p_1 \cdot b  (\mu_1 (p_1)).
 \]
 Here with bisection we mean a $t$-section $b$ such that $s\circ b$ is a diffeomorphism. \\
 In particular, if $P_2 = P_1$ (i.e.\ $P$ is the unit groupoid), Definition \ref{def_action_Lie_2_groupoid} boils down to an action of a Lie groupoid on a manifold along $\mu_2 = \mu_1$, and the automorphism above becomes the standard diffeomorphism $P \cong P$ associated to any $\CH_2$-action on $P$ and any bisection $b \in \mathrm{Bis}(\cG)$.
\end{ex}


\begin{ex}[actions of Lie groups]
If $\CH_2 = \CH_1 = G$ and $\CH_0 = \{*\}$ (Example \ref{ex_Lie_groups_are_Lie_2_groupoids}), condition 1 in Definition \ref{def_action_Lie_2_groupoid} becomes trivial while condition 2 implies the following special property: every element $g \in G$ determines a Lie groupoid isomorphism of $P_2 \tto P_1$, simply given by $(\Phi (\cdot, g), \varphi (\cdot, g))$.

In particular, if $P_2 = P_1$ (i.e.\ $P$ is the unit groupoid), Definition \ref{def_action_Lie_2_groupoid} boils down to an action of a Lie group on a manifold, and the isomorphism above becomes the standard diffeomorphism $P \cong P$ associated to any $G$-action on $P$ and any element $g \in G$.

 These are of course special instances (for $\CH_0 = \{*\}$) of the cases examined in Example \ref{ex_Lie_2_action_by_Lie_groupoid_case_2}.
\end{ex}


\begin{ex}[actions of Lie 2-groups]
If $\CH_0 = \{*\}$, i.e.\ $(\CH_2 = G_2, \CH_1 = G_1)$ is a Lie 2-group (Definition \ref{def:lie2grp}), then Definition \ref{def_action_Lie_2_groupoid} boils down to an action of a Lie 2-group $G_2 \tto G_1$ on a Lie groupoid, i.e.\ a $G_2$-action on $P_2$ and a $G_1$-action on $P_1$ such that $(\Phi,\phi)$ is a Lie groupoid morphism (since condition 1 becomes trivial). This is also called a {\it strict action} in \cite[Section 14]{CattaneoZambon13}.
\end{ex}

\begin{rk}[actions of bundles of abelian Lie groups]\label{ex_Lie_2_action_by_Lie_groupoid_case_1}
If $\CH_0 = \CH_1$ (as briefly discussed after Example \ref{ex_Lie_groupoids_are_Lie_2_groupoids} and \ref{ex_Lie_groups_are_Lie_2_groupoids}), then Definition \ref{def_action_Lie_2_groupoid} boils down to an action $\Phi$ of the bundle of abelian Lie groups $\cH_2 \tto \cH_0$ on $P_2$ which is completely 
encoded by the product $\cdot$ in $P_2$ and by a map $\phi \colon P_1 \times \CH_2 \to P_2$ taking values in the isotropy group bundle of $P_2 \tto P_1$.

Indeed, the $\CH_1$-action $\varphi$ is forced to be trivial by the commutativity of the diagram and the fact that $\gdu_{21} = \gdu_{20}$:
\[
 \varphi (p_1,g_1) = \varphi ( (\bs, s_{21}) (\bu(p_1), \gdu_{21}(g_1)) ) = \bs (\Phi (\bu (p_1), \gdu_{20}(g_1) ) ) = \bs (\bu (p_1)) = p_1.
\]
Moreover, 
\[
 \Phi (p,g) = \Phi \Big( (p, \gdu_{20}(\mu_2(p)) ) \circ (\bu (\bs(p)), g) \Big) = \Phi (p, \gdu_{20}(\mu_2(p)) ) \cdot \Phi (\bu (\bs(p)), g) = p \cdot \phi (\bs(p), g).
\]
where we set $\phi (p_1, g) := \Phi (\bu (p_1), g)$. Last, $\phi$ takes automatically values in the isotropy group bundle of $P_2 \tto P_1$.
\end{rk}


Given an action of a Lie group $G$ on a manifold $M$, each element $g \in G$ gives rise to a diffeomorphism $M \to M$. This phenomenon holds also in more generality, with actions of Lie 2-groupoids on Lie groupoids, provided we replace $g$ with a section and the group $\Diff(M)$ with a suitable group. This fact is discussed in the proposition below, which will not be used elsewhere in the paper.

\begin{prop}\label{prop_action_via_sections_and_diffeos}
	Given a Lie 2-groupoid $\CH_2\soutar\CH_1\soutar \CH_0$ acting on a Lie groupoid $P_2\soutar P_1$ with moment map $\mu: P_2 \to \cH_0$, any bisection $b\colon \CH_0\fto \CH_2$ gives rise to a diffeomorphism
\[
\phi_b: P_2 \to P_2, \quad \mathfrak{p} \mapsto \mathfrak{p} \cdot (b(\mu(\mathfrak{p})) ).
\]
Moreover,
\begin{itemize}
 \item $\phi$ is an affine morphism between Lie groupoids (see below);
 \item if $b$ takes values in the units of $\CH_2 \tto \CH_1$, then $\phi$ is an actual Lie groupoid morphism.
\end{itemize}
\end{prop}

Here with bisection we mean a $\gdt_{20}$-section $b$ such that $\gds_{20}\circ b$ is a diffeomorphism. Moreover,
we call a smooth map $\phi \colon P_2\fto Q_2$ between Lie groupoids $P_2\soutar P_1$ and $Q_2\soutar Q_1$ an {\bf affine morphism} if there is a Lie groupoid morphism $\varphi_0\colon P_2\fto Q_2$ and a $\gdt$-section $\sigma\colon Q_1\fto Q_2$ such that, for any $\mathfrak{p}\in P_2$, the following equation holds:
	$$\phi(\mathfrak{p})=\varphi_0(\mathfrak{p})\circ_{Q_2} \sigma(\gdt(\varphi_0(\mathfrak{p}))).$$ 
Affine-morphisms between Lie groupoids should not be confused with the algebraic notion of affine morphism between schemes; the adjective ``affine'' here makes reference to affine transformations on a vector space. Indeed, if two vector spaces $V$ and $W$ are seen as Lie groupoids over the point $\{*\}$, an affine morphism is an affine transformation of the kind
$$v\mapsto A(v)+w_0,$$
where $A\colon V\fto W$ is a linear map and hence a Lie groupoid morphism, and $*\mapsto w_0\in W$ is a $\gdt$-section.

\begin{proof}
First notice that $\phi_b$ is well defined because $b$ is a $\gdt_{20}$-section. It is smooth since it the composition of smooth maps, and its inverse is the smooth map
\[
\mathfrak{p} \mapsto \mathfrak{p} \cdot \tau (b ( (\gds_{20} \circ b)^{-1} ( \mu (\mathfrak{p})))).
\]

Moreover, $\phi_b$ induces another $\gdt_{20}$-section
\[
b_1\colon \CH_0\fto \CH_2, \quad g\mapsto \gdu(\gdt( b(g))) 
\]
which takes values in the units of $\CH_2\soutar\CH_1$, so that
		\begin{eqnarray*}
	\phi_b(\mathfrak{p})&=& \mathfrak{p}\cdot b(\mu(\mathfrak{p}))\\
	&=& \left(\,\mathfrak{p}\circ_{P_2} \gdu(\gds(\mathfrak{p}))\,\right)\cdot \left(\,b_1(\mu(\mathfrak{p}))\circ_{m_{21}} b(\mu(\mathfrak{p}))\,\right)\\
	&=& \left(\,\mathfrak{p}\cdot b_1(\mu(\mathfrak{p}))\,\right)\circ_{P_2} \left(\,\gdu(\gds(\mathfrak{p}))\cdot b(\mu(\mathfrak{p}))\,\right)\\
	&=& \left(\,\phi_{b_1}(\mathfrak{p})\,\right)\circ_{P_2} \left(\,\gdu(\gds(\mathfrak{p}))\cdot b(\mu(\gds(\mathfrak{p})))\,\right).
	\end{eqnarray*}
Then
\[
\sigma\colon P_1\fto P_2, \quad p\mapsto \phi_b (\gdu(p))=\gdu(p)\cdot b(\mu(p)) 
\]
is an $\gds$-section, which can be interpreted as a ``translation'' between $\phi_{b_1}$ and the map $\phi_b$. We conclude that $\phi_b$ is an affine morphism.

Last, if $b$ takes values in $\CH_1$, viewed as a submanifold of $\CH_2$, then it is a Lie groupoid morphism:
	\begin{eqnarray*}
	\phi_b(\mathfrak{p}\circ_{P_2}\mathfrak{q})&=& (\mathfrak{p}\circ_{P_2}\mathfrak{q})\cdot b(\mu(\mathfrak{p}\circ_{P_2}\mathfrak{q})) \\
	&=& (\mathfrak{p}\circ_{P_2}\mathfrak{q})\cdot \left(\, b(\mu(\mathfrak{p}))\circ_{m_{21}} b(\mu(\mathfrak{q}))\,\right)\\
		&=& \phi_b(\mathfrak{p})\circ_{P_2}\phi_b(\mathfrak{q}).      
	\end{eqnarray*}
\end{proof}

Proposition \ref{prop_action_via_sections_and_diffeos} actually combines the morphisms associated to any $\gdt_{20}$-section discussed in Examples \ref{ex_Lie_2_action_by_Lie_groupoid_case_1} and \ref{ex_Lie_2_action_by_Lie_groupoid_case_2}.

\subsection{2-Representations}

We now introduce two (related) notions of 2-representation of Lie 2-groupoids: one for 2-graded vector bundles $(E_1,E_0) \to X$ and one for anchored 2-vector bundles  $E_1\xrightarrow{\delta} E_0\fto X$.  The first one coincides with \cite[Definition 4.18]{BrahicOrtiz19}, 
while, as far we are aware, the second one is not present anywhere in the literature.

In order to compare them, it will be useful to notice that a section $\sigma\colon X\fto \GL(E_1,E_0)_0 = \Hom (E_1,E_0)$ is canonically equivalent to an anchor map $\delta_\sigma \colon E_1 \to E_0$, via the formula
$$ (\delta_{\sigma})_x \colon (E_1)_x \to (E_0)_x, \quad \quad e \mapsto \sigma(x)(e).$$

\begin{defn}\label{def_representation_2_groupoid}
A {\bf 2-graded representation of a Lie 2-groupoid} $\CH_2 \tto \CH_1 \tto \CH_0$ (Definition \ref{def_Lie_2_groupoid}) {on a 2-graded vector bundle} $(E_1, E_0)$ over $X$ (Definition \ref{def_2_vector_bundle}) is a morphism $\phi$ from the Lie 2-groupoid $\CH_i$ to the general Lie 2-groupoid $\GL(E_1,E_0)_i$ (Definition \ref{ex_general_linear_2_groupoid}):
\[\begin{tikzcd}
	{\CH_2} & {\GL(E_1,E_0)_2} \\
	{\CH_1} & {\GL(E_1,E_0)_1} \\
	\CH_0 & {\GL(E_1,E_0)_0}
	\arrow["{\phi_2}", from=1-1, to=1-2]
	\arrow["{\phi_1}", from=2-1, to=2-2]
	\arrow["{\phi_0}", from=3-1, to=3-2]
	\arrow[from=1-1, to=2-1, shift left=.5ex]
	\arrow[from=1-1, to=2-1, shift right=.5ex]
	\arrow[from=2-1, to=3-1, shift left=.5ex]
	\arrow[from=2-1, to=3-1, shift right=.5ex]
	\arrow[from=1-2, to=2-2, shift left=.5ex]
	\arrow[from=1-2, to=2-2, shift right=.5ex]
	\arrow[from=2-2, to=3-2, shift left=.5ex]
	\arrow[from=2-2, to=3-2, shift right=.5ex]
 \end{tikzcd}\]
 
 A {\bf 2-anchored representation of a Lie 2-groupoid} $\CH_2 \tto \CH_1 \tto X$ {on an anchored 2-vector bundle} $E_1\xrightarrow{\delta} E_0\fto X$ is a morphism $\phi$ from the Lie 2-groupoid $\CH_i$ to the general Lie 2-groupoid $\GL(E_1,E_0)_i$ which takes values in its restriction to the image of the section $\sigma$ corresponding to the anchor $\delta$:
 {
  \[\begin{tikzcd}
	{\CH_2} & {\GL(E_1,E_0)_2|_{\sigma(X)}} & {\GL(E_1,E_0)_2} \\
	{\CH_1} & {\GL(E_1,E_0)_1|_{\sigma(X)}} & {\GL(E_1,E_0)_1} \\
	X & {\sigma(X)} & {\GL(E_1,E_0)_0}
	\arrow["{{\phi_2}}", from=1-1, to=1-2]
 \arrow[from=1-1, to=2-1, shift left=.5ex]
 \arrow[from=1-1, to=2-1, shift right=.5ex]
	\arrow[hook, from=1-2, to=1-3]
 \arrow[from=1-2, to=2-2, shift left=.5ex]
 \arrow[from=1-2, to=2-2, shift right=.5ex]
  \arrow[from=1-3, to=2-3, shift left=.5ex]
 \arrow[from=1-3, to=2-3, shift right=.5ex]
	\arrow["{{\phi_1}}", from=2-1, to=2-2]
 \arrow[from=2-1, to=3-1, shift left=.5ex]
 \arrow[from=2-1, to=3-1, shift right=.5ex]
	\arrow[hook, from=2-2, to=2-3]
 \arrow[from=2-2, to=3-2, shift left=.5ex]
 \arrow[from=2-2, to=3-2, shift right=.5ex]
 \arrow[from=2-3, to=3-3, shift left=.5ex]
 \arrow[from=2-3, to=3-3, shift right=.5ex]
	\arrow["{{\sigma}}", from=3-1, to=3-2]
	\arrow[hook, from=3-2, to=3-3]
\end{tikzcd}\]}
\end{defn}

These two notions can be related by the following example.

\begin{ex}\label{ex_canonical_representation_of_GL}
	Let $(E_1,E_0)$ be a 2-graded vector bundle over a manifold $X$. There are two canonical 2-representations of the Lie 2-groupoid $\GL(E_1,E_0)$:
	\begin{itemize}
		\item a 2-graded representation on $(E_1,E_0)$, given by the identity morphisms;
		\item a 2-anchored representation on its associated anchored 2-vector bundle $E_1\times_X \Hom(E_1,E_0)\fto E_0\times_X \Hom(E_1,E_0)\fto \Hom(E_1,E_0)$ (see Definition \ref{ex_trivial_VB_groupoid_as_2-vector_bundle}), given by the identity morphisms.
	\end{itemize}
Notice that, while the first claim is trivial, the second one relies on the following fact: there is a canonical Lie 2-groupoid isomorphism between $\GL(E_1,E_0)$ and the restriction of $\GL \Big( E_1 \times_X \Hom(E_1,E_0), E_0 \times_X \Hom(E_1,E_0) \Big)$ to $\sigma(\Hom(E_1,E_0))$, where
$$
\sigma: \Hom (E_1,E_0) \to \GL \Big( E_1 \times \Hom(E_1,E_0), E_0 \times \Hom(E_1,E_0) \Big)_0 = \Hom(E_1,E_0) \times_X \Hom (E_1,E_0)
$$
is the section induced by the anchor map $\delta \colon E_1 \times_X \Hom(E_1,E_0) \to E_0 \times_X \Hom(E_1,E_0)$, i.e.\ the diagonal of $\Hom(E_1,E_0)$.

This example is a direct generalisation of the fact that any vector space $V$ carries a canonical representation of the Lie group $\GL(V)$, and any vector bundle $E \to X$ carries a canonical representation of the Lie groupoid $\GL(E) \tto X$.
\end{ex}

The following proposition generalises the standard characterisation of representations of Lie group(oid)s as linear actions. Notice however that it holds only for representations on anchored 2-vector bundles; for 2-graded vector bundles the only approach is the one described above as morphism in $\GL(E_1,E_0)$.

\begin{prop}\label{prop_representations_as_linear_actions}
Let $\CH_2 \tto \CH_1 \tto X$ be a Lie 2-groupoid and $E \tto E_0$ a VB-groupoid over the unit groupoid $X \tto X$. Recall that, by Lemma \ref{lemma_2_vector_bundles_2_VB_groupoids}, there is a 2-anchored vector bundle $(E_1,E_0)$ such that $E = E_1 \times_X E_0$ and the anchor $\delta$ of $(E_1,E_0)$ is the core-anchor of $E$.
	Then there is a 1-1 correspondence between
	\begin{itemize}
	 \item 2-anchored representations $(\phi_2,\phi_1,\phi_0)$ of $\CH_2 \tto \CH_1 \tto X$ on $E_1\xrightarrow{\delta}E_0\fto X$, in the sense of Definition \ref{def_representation_2_groupoid};
	\item 2-actions (in the sense of Definition \ref{def_action_Lie_2_groupoid}) of $\CH_2 \tto \CH_1 \tto X$ on the Lie groupoid $E \tto E_0$ which are linear.
	\end{itemize}
\end{prop}

\begin{proof}
 Given $\phi=(\phi_2,\phi_1,\phi_0)$, let $(\hat{\phi}_2,\hat{\phi}_1)$ be the compositions of $\phi_2$ with the projection to $\GL(E_1\times_X E_0)$, and that of $\phi_1$ with the projection to $\GL(E_0)$ respectively.
	Then $\phi$ induces a Lie 2-groupoid action on $E \tto E_0$, given by a linear action of $\CH_2$ on $E = E_1\times_X E_0$, namely
$$\CH_2 \tensor[_{\gdt_{20}}]{\times}{_{X}} (E_1\times_X E_0) \fto (E_1\times_X E_0), \quad \quad (g,v)\mapsto \hat{\phi}_2(g)v,$$
and a linear action of $\cH_1$ on $E_0$, namely
\[
 \CH_1 \tensor[_{\gdt_{10}}]{\times}{_{X}} E_0 \fto E_0, \quad \quad (g,v)\mapsto \hat{\phi}_1(g)v.
\]
The opposite direction is analogous but in the backward direction.
\end{proof}

%
%
%

\begin{ex}[representations of Lie groupoids]
When $\CH_2 = \CH_1 = \cG$ is a Lie groupoid over $X$ (Example \ref{ex_Lie_groupoids_are_Lie_2_groupoids}), then a 2-anchored representation boils down to a pair of ordinary representations of the Lie groupoid $\cG \tto X$ on the vector bundles $E_1 \times_X E_0 \to X$ and $E_0 \to X$.

As a special case, when $X = \{*\}$ and $\CH_2 = \CH_1 = G$ is a Lie group (Example \ref{ex_Lie_groups_are_Lie_2_groupoids}), a 2-anchored representation boils down to a pair of ordinary representations of the Lie group $G$, one on the vector space $E_0$ and one on the cartesian product $E_1 \times E_0$. In particular, if $E_0 = 0$ we have a single $G$-representation on $E_1$, and if $E_1 = 0$ we have two identical $G$-representations on $E_0$.
\end{ex}

\begin{ex}[representations of Lie 2-groups]
When $X = \{*\}$, i.e.\ $\CH_2 \tto \CH_1$ is a Lie 2-group, then a 2-anchored representation boils down to an ordinary representation of a Lie 2-group on a pair of vector spaces; see e.g.\ \cite[Theorem 3.5]{ShengZhu12}, which fits in the general categorical setting described in 
\cite[section 2.2.2]{Baez12}.
\end{ex}







\subsection{Principal 2-actions}

\begin{defn}\label{def_principal_Lie_2_groupoid_action}
A {\bf principal 2-action} of a Lie 2-groupoid $\CH_2 \tto \CH_1 \tto \CH_0$ on a Lie groupoid $P_2 \tto P_1$ is a 2-action (Definition \ref{def_action_Lie_2_groupoid}) such that both the $\CH_2$-action on $P_2$ and the $\CH_1$ on $P_1$ are free and proper.
\end{defn}

From the standard theory of principal groupoid bundles, since the $\CH_2$-action on $P_2$ is free and proper, its quotient $M_2:= P_2/\CH_2$ is a smooth manifold and the projection $\pi_2 \colon P_2 \to M_2$ is a surjective submersion; similarly with $M_1:= P_1/\CH_1$ and $\pi_1\colon P_1\fto M_1$. Moreover, one has the following consequence.

\begin{prop}\label{prop.cartesian.eq.rel}
Any 2-action of $\CH_2 \tto \CH_1 \tto \CH_0$ on $P_2 \tto P_1$ is principal if and only if the induced Lie groupoid morphism $(\Phi,\phi)$, as in the diagram below, is an isomorphism
\[\begin{tikzcd}
	{P_2 \tensor[_{\mu_2}]{\times}{_{t}} \CH_2} & P_2 \times_{\pi_2} P_2 \\
	{P_1 \tensor[_{\mu_1}]{\times}{_{t}} \CH_1} & P_1 \times_{\pi_1} P_1
	\arrow[from=1-1, to=2-1, shift left=.5ex, "\tis"]
	\arrow[from=1-1, to=2-1, shift right=.5ex, "\tit"']
	\arrow[from=1-2, to=2-2, shift left=.5ex]
	\arrow[from=1-2, to=2-2, shift right=.5ex]
	\arrow[from=1-1, to=1-2, "\Phi"]
	\arrow[from=2-1, to=2-2, "\phi"]
\end{tikzcd}\]
\end{prop}

\begin{proof}
 The second condition of Definition \ref{def_action_Lie_2_groupoid} implies that $(\Phi,\phi)$ is a Lie groupoid morphism, while the principality of each action is equivalent to $\Phi$ and $\phi$ being diffeomorphisms.
\end{proof}

\begin{prop}\label{prop:principal_2_action_induces_groupoid_on_quotient}
Any principal 2-groupoid action of $\CH_2 \tto \CH_1 \tto \CH_0$ on $P_2 \tto P_1$ induces a unique Lie groupoid structure on $P_2/\CH_2 \tto P_1 / \CH_1$ such that the quotient map from $P_2 \tto P_1$ to $P_2/\CH_2 \tto P_1 / \CH_1$ is a fibration of Lie groupoids.
\end{prop}

\begin{proof}
We will use the following result from \cite[Theorem 2.4.6]{Mackenzie05}. Let $\CR\subset P_2\times P_2$ be the graph of an equivalence relation $\sim_2$ in $P_2$ and $R\subset P_1\times P_1$ the graph of an equivalence relation $\sim_1$ in $P_1$. If
\begin{enumerate}
	\item $\CR\soutar R$ is a Lie subgroupoid of the cartesian product $P_2\times P_2\soutar P_1\times P_1$;
	\item the map $\CR\fto P_2 \tensor[_{\bs}]{\times}{_{\pr_1}} R, (p,q)\mapsto (p,\bs(p),\bs(q))$ is a surjective submersion;
\end{enumerate}
then there is a unique Lie groupoid structure in $(P_2/\sim_2)\soutar (P_1/\sim_1)$, with the quotient map being a fibration. If $\sim_2$ and $\sim_1$ are given by the quotients by the $\cH_2$ and $\cH_1$-actions, then the condition (1) is satisfied by Proposition \ref{prop.cartesian.eq.rel}, while condition (2) is equivalent to the map $P_2 \tensor[_{\mu_2}]{\times}{_{\bt}} \fto P_2 \tensor[_{\mu_1 \circ \bs}]{\times}{_{\bt}} \CH_1, (p,g)\mapsto (p,\bs(g))$ being a surjective submersion, which is clearly true.
\end{proof}

\begin{ex}\label{ex_principal_action_Lie_2_groupoid_on_itself}
 Let $\CH_2 \tto \CH_1 \tto \CH_0$ be a Lie 2-groupoid. Its 2-action on itself from Example \ref{ex_action_Lie_2_groupoid_on_itself} is principal, since the (1-)actions of $\CH_2 \tto \CH_0$ and of $\CH_1 \tto \CH_0$ on themselves are principal. The quotient of both groupoid actions is isomorphic to $\CH_0$, so that the quotient of the principal 2-groupoid action is the unit groupoid $\CH_0 \tto \CH_0$.
\end{ex}

\subsection{PB-groupoids}\label{sec_PB_groupoids}

\begin{defn}\label{def_PB_groupoid}
A {\bf PB-groupoid} with structural 2-groupoid $\CH_2 \tto \CH_1 \tto \CH_0$ is a commutative diagram
\[\begin{tikzcd}
{P_\cG} & \cG \\
{P_M} & M
\arrow[from=1-1, to=2-1, shift left=.5ex, "\bs"]
\arrow[from=1-1, to=2-1, shift right=.5ex, "\bt"']
\arrow[from=1-2, to=2-2, shift left=.5ex, "s"]
\arrow[from=1-2, to=2-2, shift right=.5ex, "t"']
\arrow["{\Pi_\cG}", from=1-1, to=1-2]
\arrow["{\Pi_M}", from=2-1, to=2-2]
\end{tikzcd}\]
such that
\begin{itemize}
 \item $P_\CG\soutar P_M$ and $\CG\soutar M$ are Lie groupoids;
 \item $\Pi_\cG$ and $\Pi_M$ are surjective submersions;
 \end{itemize}
together with a principal 2-action (Definition \ref{def_principal_Lie_2_groupoid_action}) of $\CH_2 \tto \CH_1 \tto \CH_0$ on $P_\CG\soutar P_M$ and a Lie groupoid isomorphism $\varphi$ between $(P_\CG/\CH_2)\soutar (P_M/\CH_1)$ (see Proposition \ref{prop:principal_2_action_induces_groupoid_on_quotient}) and  $\CG\soutar M$ such that $\Pi_\cG =\varphi\circ \pr$, where $\pr \colon P_\CG\fto P_\CG/\CH_2$ denotes the quotient map.
\end{defn}

Notice that $\Pi_\cG$ is a Lie groupoid fibration by Proposition \ref{prop:principal_2_action_induces_groupoid_on_quotient}, hence, in particular, it is a Lie groupoid morphism.

\begin{ex}[PB-groupoids with structural Lie 2-groups]
 If the structural Lie 2-groupoid is a Lie 2-group (Definition \ref{def:lie2grp}), then a PB-groupoid in the sense of Definition \ref{def_PB_groupoid} boils down to a \textit{principal bundle groupoid} in the sense of \cite[Definition 1.15]{GarmendiaPaycha23} and to a \textit{principal 2-bundle over a Lie groupoid} in the sense of  \cite[Definition 3.4]{CCK22} and of \cite[Definition 5.9]{HerreraOrtiz23}; in this last work, the definition is formulated via the equivalent condition from Proposition \ref{prop.cartesian.eq.rel}. A version of these objects involving differentiable stacks was considered in \cite[Definition 3.22]{BursztynNosedaZhu20}.
 
 \textit{Principal 2-bundles} as in \cite[Definition 6.1.5]{NW} consist of a Lie 2-group $\CH_2\soutar \CH_1$ acting on a Lie groupoid $P_\CG\soutar P_M$, together with a map $P_\CG\fto Y$ such that the Lie groupoids $\CH_2 \times P_\CG \tto P_\cG$ and $P_\CG\times_Y P_\CG \tto P_\cG$ are Morita equivalent. As proved in \cite[Proposition 1.32]{GarmendiaPaycha23}, these objects are the same things as PB-groupoids with structural Lie 2-group $\CH_2\soutar \CH_1$ over fibred pair groupoid $\CG=M\times_Y M\soutar M$, for a submersion $M\fto Y$.

Last, note that \textit{semi-strict principal 2-bundles} from \cite[Definition 2.9]{Wockel11} are defined in terms of local trivialisations and cannot be directly comparable with (particular cases of) our Definition \ref{def:lie2grp} since they involve actions of Lie 2-groups on \textit{smooth 2-spaces} (objects more general than a Lie groupoid). 
%
\end{ex}

\begin{ex}[PB-groupoids with structural Lie groupoids]\label{ex.groupoid.act}

%
%

 If the structural Lie 2-groupoid is a Lie groupoid, with $\CH_2 = \CH_1$ (Example \ref{ex_Lie_groupoids_are_Lie_2_groupoids}), then a PB-groupoid in the sense of Definition \ref{def_PB_groupoid} yields a pair of standard $(\CH_2 \tto \CH_0)$-principal bundles $(P_\cG \to \cG, P_M \to M)$ (in the sense of e.g.\ \cite[Section 5.7]{MoerdijkMrcun03}, \cite[Definition 2.1]{R}, \cite[Definition 3.1]{I} or \cite[Section 2.2]{H}). In particular, when $\CH_2\tto\CH_0$ is a bundle of Lie groups and $\CG=M$ we recover  \textit{generalized principal bundles} as in \cite[Theorem 2.1]{CastrillonRodriguez23}.
\end{ex}

\begin{ex}[PB-groupoids with structural Lie groups]

%
%
%
 If the structural Lie 2-groupoid is of the form $H\soutar H\soutar \{*\}$, then a PB-groupoid boils down to a \textit{principal bundle over a groupoid} in the sense of \cite[Definition 2.33]{LTX}, to a \textit{PBG-groupoid} in the sense of \cite[Definition 2.1]{Mackenzie87paper} and to a  \textit{$G$-groupoid} in the sense of \cite[Definition 4.4]{BruceGrabowskaGrabowski17}. This is a particular case, for $\CH_0=\{*\}$, of Example \ref{ex.groupoid.act}.

 Last, note that {\it $G$-principal groupoids} from \cite[Definition 2.2]{FernandesStruchiner23} cannot be directly comparable with (particular cases of) our Definition \ref{def_PB_groupoid}, since they involve principal actions of Lie groups on the total space of Lie groupoids, which induce non-principal actions on the base space.
\end{ex}

\begin{ex}[Lie 2-groupoids as PB-groupoids]\label{ex_lie_2_groupoids_as_PB_groupoids}
Let $\CH_2 \tto \CH_1 \tto \CH_0$ be a Lie 2-groupoid. Then $\CH_2 \tto \CH_1$ defines a PB-groupoid over the unit groupoid $\CH_0 \tto \CH_0$, with $\Pi_\cG$ and $\Pi_M$ given, respectively, by $t_{20}$ and $t_{10}$, and with the principal action on itself (see Example \ref{ex_principal_action_Lie_2_groupoid_on_itself}).
\end{ex}

\subsection{First half of the correspondence}

Let $\cH \tto X$ be a Lie groupoid. Recall that, given a principal $\cH$-bundle $P \to M$, with moment map $\mu \colon P \to X$, and a $\cH$-representation $E \to X$, then its {\bf associated bundle} is the quotient
\[
P[E] := (P \times_X E)/\cH,
\]
which is a vector bundle over $M$ with rank equal to the rank of $E$. This construction generalises the standard associated bundles for principal groups bundles, and can in turn be generalised to PB-groupoids.

\begin{thm}\label{thm_from_PB_to_VB}
Let $P_{\cG} \tto P_M$ be a PB-groupoid with structural Lie 2-groupoid $\cH_2 \tto \cH_1 \tto  X$ (Definition \ref{def_PB_groupoid}), and $E_1 \to E_0 \to X$ a 2-anchored representation of $\cH$ (Definition \ref{def_representation_2_groupoid}), which we view as a linear action of $\cH$ on $E_1\times_X E_0 \tto E_0$, as in Proposition \ref{prop_representations_as_linear_actions}. Then the quotients
\[
P[E]_2 := (P_{\cG} \times_{X} (E_1\times_X E_0))/\cH_2, \quad \quad P[E]_1 := (P_M \times_{X} E_0)/\cH_1
\]
define a VB-groupoid
\[\begin{tikzcd}
	P[E]_2 & \cG \\
	P[E]_1 & M
	\arrow[from=1-1, to=2-1, shift left=.5ex, "\tis"]
	\arrow[from=1-1, to=2-1, shift right=.5ex, "\tit"']
	\arrow[from=1-2, to=2-2, shift left=.5ex, "s"]
	\arrow[from=1-2, to=2-2, shift right=.5ex, "t"']
	\arrow["{\pi_\cG}", from=1-1, to=1-2]
	\arrow["{\pi_M}", from=2-1, to=2-2]
\end{tikzcd}\]
with rank $(\rank(E_1),\rank(E_0))$.
\end{thm}

In analogy with the classical construction for principal bundles with structural groups/groupoids, $P[E]_2 \tto P[E]_1$ is called the {\bf associated bundle}.

\begin{proof}
 It is known that $P[E]_2 \to \cG$ and $P[E]_1 \to M$ are vector bundles of ranks, respectively, equal to the ranks of $E_1\times_M E_0 \to X$ and $E_0 \to X$. The groupoid structure on $P[E]_2 \tto P[E]_1$ is induced by that of the product Lie groupoid  $P_{\cG} \times_{X} (E_1\times_X E_0) \tto (P_M \times_{X} E_0)$, the free and proper action of $\cH_2$ and of $\cH_1$, and Proposition \ref{prop:principal_2_action_induces_groupoid_on_quotient}. 
\end{proof}

\section{From VB-groupoids to PB-groupoids}

\subsection{Frames adapted to a VB-groupoid}


The entire frame bundle of a VB-groupoid is too big to support the structure of a PB-groupoid (or even only of a Lie groupoid). In this section we are going to identify the class of frames which are ``adapted'' to the VB-groupoid structure, and which will give rise to a PB-groupoid.

We write here, for later use, the following formula of the ``right-translation'' by elements $g \in \cG$ in an arbitrary VB-groupoid $E_\cG \tto E_M$:
\begin{equation}\label{eq_right_translation}
 \tiR_g := \tim (\cdot, 0_g): \ker(\tis)|_{s^{-1}(t(g))} \to \ker(\tis)|_{s^{-1}(s(g))}.
\end{equation}
This is simply the vector bundle isomorphism which generalises the differential of the right translation $R_g \colon s^{-1} (t(g)) \to s^{-1} (s(g))$ for the tangent VB-groupoid $T\cG \tto TM$.

The main ingredient we will use to define adapted frames (and to prove that they have a groupoid structure) is the notion of fat groupoid, introduced in \cite[section 8]{GraciaMehta17} and briefly recalled below.
 
\begin{defn}\label{def_fat_groupoid}
Let $E_\cG \tto E_M$ be a VB-groupoid of rank $(l,k)$ over $\cG \tto M$. Its {\bf fat groupoid} $\tilde{\cG}(E_\cG)$ consists of all the pairs $(g,H_g)$, with $g \in \cG$ and $H_g \subseteq (E_\cG)_g$ any (necessarily $k$-dimensional) subspace such that
\[
 H_g \oplus \ker(\tis_g) = (E_\cG)_g, \quad \quad H_h \oplus \ker(\tit_g) = (E_\cG)_g. \qedhere
\]
\end{defn}

The set $\tilde{\cG}(E_\cG)$ has a natural Lie groupoid structure over $M$, with structure maps
\[
 s (g,H_g) = s(g), \quad \quad t (g,H_g) = t(g),
\]
\[
 m \Big( (g,H_g), (h,H_h) \Big) = \Big (gh, \tim (H_g,H_h):= \big\{ \tim (v_g,v_h) | v_g \in H_g, v_h \in H_h, \tis (v_g) = \tit (v_h) \big\} \Big).
\]
\[
 u (x) = \Big( 1_x, H_{1_x}:= \tiu (E_M)_x \Big), \quad \quad i (g,H_g) = \Big( g^{-1}, H_{g^{-1}}:= \tii (H_g) \Big).
\]
Moreover, there is a natural representation of $\tilde{\cG}(E_\cG) \tto M$ on the vector bundle $E_M \to M$, given by
\begin{equation}\label{eq_representation_fat_groupoid_on_base}
 (g,H_g) \cdot e_x := \tit_g (\tis_g|_{H_g})^{-1} (e_x),
\end{equation}
and a similar one on $C \to M$, given by
\begin{equation}\label{eq_representation_fat_groupoid_on_core}
 (g,H_g) \cdot c_x := \alpha_g \circ_{E_\cG} c_x \circ_{E_\cG} 0_{g^{-1}} = \tim (\alpha_g, \tiR_{g^{-1}} (c_x) ),
\end{equation}
for $\alpha \in H_g$ such that $\tis_g (\alpha) = \tit_{1_x} (c_x)$.

\begin{ex}[trivial base and trivial core]\label{ex_fat_groupoid_trivial_base_core}
 If $l = 0$, then the Lie groupoid $\tilde{\cG}(E_\cG) \tto M$ is isomorphic to $\cG \tto M$: indeed, since $E_\cG$ has rank $l+k$ and the spaces $H_g$ are $k$-dimensional, we are forced to impose $H_g = (E_\cG)_g$. Note also that the representation \eqref{eq_representation_fat_groupoid_on_base} coincides precisely with the representation of $\cG$ on $E_M$ from Example \ref{ex_VB_groupoid_trivial_core}.

 Similarly, if $k = 0$, then $\tilde{\cG}(E_\cG) \cong \cG$ as well, since each $H_g$ is forced to be the zero vector space.
 Also in this case, the representation \eqref{eq_representation_fat_groupoid_on_core} boils down to the representation of $\cG$ on $C$ from Example \ref{ex_VB_groupoid_trivial_base}.
\end{ex}

We can finally define the frames adapted to a VB-groupoid.

\begin{defn}\label{def_special_frames}
 Let $E_\cG \tto E_M$ be a VB-groupoid of rank $(l,k)$ over $\cG \tto M$. A frame $\phi_g \in \Fr(E_{\cG})$ is called an {\bf \frs\ frame} if the following two conditions holds:
 \begin{enumerate}
 \item $\phi_g|_{\RR^l \times \{0\}}$ takes values in $\ker(\tis_g)$;
 \item the image $H_g \subseteq (E_\cG)_g$ of $\phi_g|_{\{0\} \times \RR^k}$ is an element of the fat groupoid $\tilde{\cG}(E_\cG)$ (Definition \ref{def_fat_groupoid})
 \end{enumerate}
 The collection
 \[
 \Fr(E_\cG)^{\rm \sbis} := \{ \phi_g \colon \RR^{l+k} \to (E_\cG)_g \mid \phi_g \text{ \frs\ frame} \} \subseteq \Fr(E_\cG).
\]
of \frs\ frames of $E_\cG$ is called the \textbf{\frs\ frame bundle} of $E_\cG$.
 \end{defn}

{
\begin{rk}\label{lemma_automatic_s_section}
	The term ``\frs\ frames'' in Definition \ref{def_special_frames} is motivated by the following reasons.
\begin{itemize}
 \item The ``bisection'' part refers to the second condition and points to an analogy with the definition of bisections. Indeed, a bisection of a Lie groupoid $\cG$ can be equivalently described as a submanifold $\Sigma \subseteq T\cG$ such that $(g,T_g \Sigma)$ belongs to the fat groupoid $\tilde{\cG} (E_\cG)$ for every $g \in \Sigma$.
\item The symbol $s$ refers to the first condition. Of course, the choice of $\ker(\tis)$ instead of $\ker(\tit)$ is purely a convention; indeed, later on we will also use \frt\ frames (see Definition \ref{def_special_frames_t}).

\end{itemize}

Notice also that condition 1 in Definition \ref{def_special_frames} implies that $H_g$ is automatically transverse to $\ker(\tilde s)$; therefore, condition 2 can be simplified to $H_g$ being transverse to  $\ker(\tilde t)$.
\end{rk}
}

The key feature of \frs\ frames is the fact that they induce two frames of the base vector bundle $E_M \to M$ (Lemma \ref{lemma_equivalent_def_bisection_frames}) and two frames of the core $C \to M$ (Lemmas \ref{lemma_equivalent_def_frame_taking_values_in_kers} and \ref{prop_induced_s-frame_A}).

\begin{lem}\label{lemma_equivalent_def_bisection_frames}
 Let $E_\cG \tto E_M$ be a VB-groupoid of rank $(l,k)$ over $\cG \tto M$ and $\varphi_g \colon \RR^k \to (E_\cG)_g$ be an injective linear map. The following two properties are equivalent:
 \begin{enumerate}
  \item the image $\Ima(\varphi_g)$ is transverse to $\ker(\tis)$, i.e.\ $\ker(\tis_g)\oplus \Ima( \varphi_g) = (E_\cG)_g$;
  \item the composition
\[
\tis \circ \varphi_g \colon \RR^k \to (E_M)_{s(g)}
\]
is a frame of $E_M \to M$.
 \end{enumerate}
 The same statement holds replacing $\tis$ with $\tit$ everywhere.

In particular, any \frs\ frame $\phi_g: \RR^{l+k} \to E_\cG$ induces the following frames of $E_M \to M$:
 \[ 
 \bs(\phi_g)^b: \RR^k \to (E_M)_{s(g)}, \quad \quad \bs(\phi_g)^b (v) := \tis (\phi_g (0,v) ),
\]
 \[ 
 \bt(\phi_g)^b: \RR^k \to (E_M)_{t(g)}, \quad \quad \bt(\phi_g)^b (v) := \tit (\phi_g (0,v) ).
\]
\end{lem}

\begin{proof}
 \underline{1 $\Rightarrow$ 2}: take $v_1,v_2 \in \RR^k$ such that $\tis (\varphi _g (v_1) ) = \tis (\varphi_g (v_2))$; they the hypothesis yields
 \[
\varphi_g (v_1-v_2) \in \Ima( \varphi_g) \cap \ker(\tis_g) = \{0\}.
 \]
 As $\phi_g$ is injective, this implies that $v_1 = v_2$, i.e.\ $\tis \circ \varphi_g$ is injective. Since $\rank(E_M) = k = \dim(\Ima(\phi_g))$, we conclude that $\tis \circ \phi_g$ is a linear isomorphism, i.e.\ a frame of $E_M \to M$.

 \underline{2 $\Rightarrow$ 1}: consider any element $\varphi_g (v)$ in the intersection $\ker(\tis_g) \cap \Ima( \varphi_g)$. Since $\tis \circ \phi_g$ is injective, the condition $\tis (\varphi_g (v))= 0$ implies $v =0$, i.e.\ the intersection $\ker(\tis_g) \cap \Ima( \varphi_g)$ contains only the zero vector. It follows that
\[
 \dim ( \ker(\tis_g) \oplus \Ima( \varphi_g ) ) = \dim(\ker(\tis_g)) + \dim (\Ima( \varphi_g )) = l + k = \rank (E_\cG),
\]
 hence $\ker(\tis_g) \oplus \Ima( \varphi_g ) = (E_\cG)_g$.
\end{proof}

\begin{lem}\label{lemma_equivalent_def_frame_taking_values_in_kers}
 Let $E_\cG \tto E_M$ be a VB-groupoid of rank $(l,k)$ over $\cG \tto M$ and $\varphi_g \colon \RR^l \to (E_\cG)_g$ be an injective linear map. The following two properties are equivalent:
\begin{enumerate}
 \item $\varphi_g$ takes values in $\ker(\tis_g)$;
 \item the composition
\[
 \tiR_{g^{-1}} \circ \varphi_g \colon \RR^l \to C_{t(g)} \subseteq (E_\cG)_{1_{t(g)}},
 \]
with $\tiR$ given by \eqref{eq_right_translation}, is well-defined and is a frame of $C \to M$.
\end{enumerate}
%
The same holds replacing $\tis$ with $\tit$ everywhere.

In particular, any \frs\ frame $\phi_g: \RR^{l+k} \to E_\cG$ induces the following frame of $C \to M$:
 \[ 
 \bt(\phi_g)^c: \RR^l \to C_{t(g)}, \quad \quad \bt(\phi_g)^b (v) := \tiR_{g^{-1}} (\phi_g (w,0) ).
 \]
\end{lem}

\begin{proof}

 \underline{1 $\Rightarrow$ 2}: since $\varphi_g$ is an injective linear map, taking values in the $l$-dimensional space $\ker(\tis_g)$, it is a linear isomorphism. Then its composition with the linear isomorphism $\tiR_{g^{-1}}$ is well-defined and is a frame of $C \to M$.
 



\underline{2 $\Rightarrow$ 1}: since the composition $\tiR_{g^{-1}} \circ \varphi_g$ is well defined, $\varphi_g$ must necessarily takes values in $\ker(\tis_g) \subseteq (E_\cG)_g$.
\end{proof}

\begin{defn}\label{def_t_phi}
 For any \frs\ frame $\phi_g \in \Fr^{\sbis} (E_\cG)$ we denote by
\[
\dphig \colon \RR^{l} \to \RR^k
\]
the unique linear map satisfying the following diagram:
\[\begin{tikzcd}
	{C_{t(g)}} & {(E_M)_{t(g)}} \\
	{\RR^{l}} & {\RR^k}
	\arrow["{\dphig}", dashed, from=2-1, to=2-2]
	\arrow["{\rho_{t(g)}}", from=1-1, to=1-2]
	\arrow["{\bt(\phi_g)^b}"', from=2-2, to=1-2]
	\arrow["{\bt(\phi_g)^c}", from=2-1, to=1-1]
	\arrow["\cong"', from=2-1, to=1-1]
	\arrow["\cong", from=2-2, to=1-2]
\end{tikzcd}\]

Equivalently, since $\bt(\phi_g)^b$ is an isomorphism by Lemma \ref{lemma_equivalent_def_bisection_frames}, $\dphig$ is simply the composition of linear maps
\[
 \dphig := (\bt(\phi_g)^b)^{-1} \circ \rho_{t(g)} \circ \bt(\phi_g)^c. \qedhere
\]
\end{defn}
In particular, $\dphig$ vanishes if and only if $\phi_g (0,w) \in \ker (\tit_g)$ for every $w \in \RR^l$: the map $\dphig$ measures then how non-zero the element $\tit (\phi_g (w,0))$ can be.

\begin{lem}\label{prop_induced_s-frame_A}
 Any \frs\ frame $\phi_g \colon \RR^{l+k} \to (E_\cG)_g$ induces the following frame of $C \to M$:
\[
 \bs(\phi_g)^c: \RR^{l} \to C_{s(g)}, \quad \quad \bs(\phi_g)^c (w) := \tiR_g (\tii (\phi_g ( -w, \dphig(w))) ),
\] 
with $\tiR_g$ given by equation \eqref{eq_right_translation} and $\dphig$ from Definition \ref{def_t_phi}.
\end{lem}

\begin{proof}
The vector $\bs(\phi_g)^c (w) \in (E_\cG)_g$ takes values in $\ker(\tis_{1_{s(g)}}) = C_{s(g)}$. Indeed, Lemma \ref{lem_t_phi} yields $\phi_g (-w, \dphig(w)) \in \ker(\tit_g)$, and after applying $\tii$ and $\tiR_g$ one obtains an element in $\ker(\tis_{1_{s(g)}})$.

Moreover, we claim that $\bs(\phi_g)^c$ is a linear isomorphism. Since $C$ has rank $l$, it is enough to check its injectivity. If $\bs(\phi_g)^c (w) = 0$, then condition 1 of Definition \ref{def_special_frames} implies that
\[
 0 = \tit (\bs(\phi_g)^c (w) ) = \tit (\tii (\phi_g (-w,\dphi (w)) ) = \tis (\phi_g (-w,\dphi (w)) ) = \tis (\phi_g (0,\dphi (w)) ).
\]
As a consequence, the element $\phi_g (0, \dphi (w)) \in \Ima(\phi_g|_{\{0\} \times \RR^k} )$ belongs also to $\ker(\tis)$ and is therefore zero by Lemma \ref{lemma_equivalent_def_bisection_frames}, as the image of $\phi_g|_{\{0\} \times \RR^k}$ is an element of the fat groupoid. Since $\phi_g$ is injective, one has $\dphi (w) = 0$. Plugging this in $\bs(\phi_g)^c (w) = 0$ we get
\[
0 = \tiR_g (\tii (\phi_g (w,0) ) )
\]
which implies $\phi_g (w,0) = 0$; by the injectivity of $\phi_g$, we conclude that $w = 0$.
\end{proof}

\subsection{The groupoid structure of the \frs\ frames}

The first part of the main result of this section is the Lie groupoid structure on the space of adapted frames:
\[
\Fr(E_\cG)^{\rm \sbis} \tto \Fr(C) \times_M \Fr(E_M).
\]
We will focus here on the algebraic part (i.e.\ the groupoid structure), postponing the smoothness to Theorem \ref{thm_special_frames_are_principal_bundles} later on.


\begin{thm}\label{thm_special_frames_are_groupoid}
Given a VB-groupoid $E_\cG \tto E_M$ of rank $(l,k)$, its \frs\ frame bundle $\Fr(E_\cG)^{\rm \sbis}$ is a groupoid over $\Fr(C) \times_M \Fr(E_M)$, where $C \to M$ is the core of $E_\cG \tto E_M$ (Definition \ref{def_core}), with structure maps induced by those of $E_\cG \tto E_M$:
\[
 \bs \colon \Fr(E_\cG)^{\rm \sbis} \to \Fr(C) \times_M \Fr(E_M), \quad \bs (\phi_g) (w,v) = \Big(\bs(\phi_g)^c (w):= \tiR_g ( \tii (\phi_g (-w,\dphig( w ))) )
 ,  \bs(\phi_g)^b (v) := \tis (\phi_g (0,v) ) \Big);
\]
\[
 \bt \colon \Fr(E_\cG)^{\rm \sbis} \to \Fr(C) \times_M \Fr(E_M), \quad \bt (\phi_g) (w,v) = \Big( \bt(\phi_g)^c (w) := \tiR_{g^{-1}} ( \phi_g (w,0) ), \bt(\phi_g)^b (v) := \tit (\phi_g (0,v) ) \Big);
\]
\[
 \bm \colon \Fr(E_\cG)^{\rm \sbis} \tensor[_{\bs}]{\times}{_{\bt}} \Fr(E_\cG)^{\rm \sbis}\to \Fr(E_\cG)^{\rm \sbis}, \quad \bm (\phi_g, \phi_h ) (w,v) := \tim \Big( \phi_g (0,v), \phi_h (0,v) \Big) + \tiR_h ( \phi_g (w,0));
\]
\[
 \bu \colon \Fr(C) \times_M \Fr(E_M) \to \Fr(E_\cG)^{\rm \sbis}, \quad \bu (\varphi_x^c, \varphi^b_x) (w,v) := \tiu (\varphi^b_x ( v ) ) + \varphi^c_x (w);
\]
\[
 \bi \colon \Fr(E_\cG)^{\rm \sbis} \to \Fr(E_\cG)^{\rm \sbis}, \quad \bi (\phi_g) (w,v):= \tii (\phi_g (-w,v + \dphig(w)) ). \qedhere
\]
\end{thm}


\begin{lem}\label{lem_t_phi}
The map $\dphig$ from Definition \ref{def_t_phi} satisfies the following properties:
\begin{enumerate}
\item $\dphig=d_{\bu \bt \phi_g}=d_{\bu \bs \phi_g}$;
\item if $\phi_g$ and $\phi_h$ are composable, then $d_{\phi_g}=d_{\phi_h}=d_{\bm(\phi_g,\phi_h)}$.
\end{enumerate}
\end{lem}

\begin{proof}
It is also clear that 1. implies 2.; moreover, by definition of $\dphig$, one has $\dphig=d_{\bu \bt \phi_g}$. We only need to prove $d_{\phi_g}=d_{\bu \bs \phi_g}$:
\begin{eqnarray*}
d_{\bu\bs\phi_g} (w) 
&=& \Big( (\bs(\phi_g)^b)^{-1} \circ \tit \circ \bs(\phi_g)^c \Big) (w) \\
&=& (\bs(\phi_g)^b)^{-1} \Big( \tit (\tiR_{g^{-1}} (\tii (\phi_g (-w,\dphig(w)) ) )) \Big) \\
&=& (\bs(\phi_g)^b)^{-1} \Big( \tis (\phi (-w,\dphig(w)) ) \Big) \\
&=& (\bs(\phi_g)^b)^{-1} \Big( \tis (\phi_g (0,\dphig(w)) ) - \cancel{\tis (\phi_g (w,0) )} \Big) \\
&=& \dphig(w). 
\end{eqnarray*} \qedhere
\end{proof}

The proof of Theorem \ref{thm_special_frames_are_groupoid} is based on the fact that $\Fr^\sbis(E_\cG)$ can be described by means of its fat groupoid $\tilde{\cG}(E_\cG)$ (Definition \ref{def_fat_groupoid}). Notice first that the representations \eqref{eq_representation_fat_groupoid_on_base} and \eqref{eq_representation_fat_groupoid_on_core} induce a left action of $\tilde{\cG}(E_\cG)$ on $\Fr(C) \to M$ and $\Fr(E_M) \to M$, hence an action (which we denote $\mathrm{act}$) of $\tilde{\cG}(E_\cG)$ on $\Fr(C) \times_M \Fr(E_M) \to M$. One can therefore consider the action groupoid
\[
 \tilde{\cG}(E_\cG) \times_M (\Fr (C) \times_M \Fr(E_M)) \tto \Fr (C) \times_M \Fr(E_M),
\]
with source $\pr$ and target $\mathrm{act}$.

\begin{lem}\label{lemma_isomorphism_action_groupoid}
Let $E_\cG$ be a VB-groupoid. The function
\[
F: \Fr^\sbis(E_\cG) \to \tilde{\cG}(E_\cG) \times_M (\Fr (C) \times_M \Fr(E_M)), \quad \quad \phi_g \mapsto \Big( \big( g,H_g:= \Ima (\phi_g (0,\cdot)) \big), \big( \bs (\phi_g)^c, \bs(\phi_g)^b \big) \Big)
\]
is a bijection such that
\[
\bs (\phi_g) = \pr (F (\phi_g)), \quad \quad \bt (\phi_g) = \mathrm{act} (F (\phi_g)) \quad \quad \forall \phi_g \in \Fr^\sbis(E_\cG).
\]
\end{lem}

\begin{proof}
The function $F$ is well defined: indeed, the pair $(g,H_g)$ belongs to the fat groupoid by definition of bisection frame, while $\bs(\phi_g)^c$ and $\bs (\phi_g)^b$ are frames of $C$ and $E_M$, respectively, by Lemma \ref{prop_induced_s-frame_A} and \ref{lemma_equivalent_def_bisection_frames}. Furthermore, $F$ is a bijection, with inverse given by
\[
 F^{-1} \Big( (g,H_g),(\varphi^c,\varphi^b) \Big) (w,v) := (\tis|_{H_g})^{-1} (\varphi^b (v) ) + (\tis|_{H_g})^{-1} (\tit (\varphi^c (w))) - \tii (\tiR_g)^{-1} (\varphi^c (w)).
\]
The definition of $\pr$ and of $F$ yields directly the first condition, i.e.\
 \[
  \pr (F(\phi_g)) = \big( \bs (\phi_g)^c, \bs(\phi_g)^b \big) = \bs (\phi_g)
 \]
To check the second condition, we have to compute
\[
 \mathrm{act} (F(\phi_g)) = \Big( (g,H_g) \cdot \bs (\phi_g)^c, (g,H_g) \cdot \bs(\phi_g)^b \Big).
\]
Using \eqref{eq_representation_fat_groupoid_on_base}, the second component becomes
\begin{eqnarray*}
(g,H_g) \cdot \bs(\phi_g)^b (\cdot) &=& (g,H_g) \cdot \tis_g (\phi_g (0,\cdot) ) \\
&=& \tit_g \circ (\tis_g|_{\Ima (\phi_g (0,\cdot))})^{-1} \circ \tis_g \circ \phi_g|_{\{0\} \times \RR^k} (\cdot) \\
&=&  \tit_g (\phi_g (0, \cdot) ) = \bt(\phi_g)^b (\cdot).
\end{eqnarray*}
For the first component, first we notice that, for every $w \in \RR^l$,
\[
 \alpha_g := \phi_g (0,d_{\phi_g}(w))
\]
belongs to $H_g = \Ima (\phi_g|_{\{0\} \times \RR^k})$ and satisfies
\begin{eqnarray*}
 \tit (\bs(\phi_g)^c (w)) &=& \tit (\tiR_g \circ \tii \circ \phi_g  (-w, d_{\phi_g} w)) \\
 &=& \tit (\tii \circ \phi_g  (-w, d_{\phi_g} w) ) \\
 &=& \tis (\phi_g  (-w, d_{\phi_g} w) ) \\
 &=& \tis (\phi_g  (0, d_{\phi_g} w) ) + \cancel{\tis (\phi_g  (-w, 0) )} \\
 &=& \tis(\alpha_g).
\end{eqnarray*}
Furthermore, using the interchange law,
\begin{eqnarray*}
 \bs(\phi_g)^c (w) &=& \tiR_g (\tii (\phi_g (-w, d_{\phi_g} (w))) ) \\
 &=& \tim \Big( \tii \big(\phi_g (-w, d_{\phi_g} (w)) \big), 0_g \Big) \\
 &=& \tim \Big( \tii \big(\phi_g (0, d_{\phi_g} (w)) \big) + \tii \big( \phi_g (-w,0) \big), \phi_g (w,0) + \phi_g (-w,0) \Big) \\
 &=& \tim \Big( \tii \big(\phi_g (0, d_{\phi_g} (w)) \big), \phi_g (w,0) \Big) + \tim \Big( \tii \big( \phi_g (-w,0) \big), \phi_g (-w,0) \Big) \\
&=& \tii( \alpha_g ) \circ \phi_g (w,0) + \tiu \big( \tis (\phi_g (-w,0)) \big) \\
&=& \tii( \alpha_g ) \circ \phi_g (w,0) + 0_{1_{s(g)}} = \tii (\alpha_g ) \circ \phi_g (w,0).
\end{eqnarray*}
Therefore, using \eqref{eq_representation_fat_groupoid_on_core},
\begin{eqnarray*}
(g,H_g) \cdot \bs(\phi_g)^c (w) &=& 
\alpha_g \circ \bs(\phi_g)^c (w) \circ 0_{g^{-1}} \\
&=& \alpha_g \circ \tii (\alpha_g) \circ \phi_g (w,0) \circ 0_{g^{-1}} \\
&=& 0_{1_{t(g)}} \circ \phi_g (w,0) \circ 0_{g^{-1}} \\
&=&  \tiR_{g^{-1}} \circ \phi_g (w, 0) = \bt(\phi_g)^c (w).
\end{eqnarray*}
\end{proof}

\begin{proof}[Proof of Theorem \ref{thm_special_frames_are_groupoid}]

In general, given three sets $\cG_1$, $\cG_2$ and $M$, and four maps $s_1,t_1: \cG_1 \to M$, $s_2,t_2: \cG_2 \to M$, if there is a bijection between $\cG_1$ and $\cG_2$ which preserves the maps $s_1,s_2$ and $t_1,t_2$, there we can transport a groupoid multiplication on $\cG_2$ (with source $s_2$ and target $t_2$) to one on $\cG_1$ (with source $s_1$ and target $t_1$). The statement of the Theorem then follows from Lemma \ref{lemma_isomorphism_action_groupoid}, by transporting the structure maps of the action groupoid to $\Fr(E_\cG)^{\rm \sbis}$. Indeed, the multiplication satisfies
\begin{eqnarray*}
 F (\bm (\phi_g, \phi_h)) &=& F \Big( \tim (\phi_g|_{\{0\} \times \RR^k }, \phi_h|_{\{0\} \times \RR^k }) + \tiR_h (\phi_g|_{\RR^l \times \{0\} }) \Big) \\
 &=& \Big( \big( gh, \Ima (\tim (\phi_g|_{\{0\} \times \RR^k }, \phi_h|_{\{0\} \times \RR^k }))  \big), \big( \bs ( \bm (\phi_g,\phi_h) )^c, \bs ( \bm (\phi_g,\phi_h) )^b \big) \Big) \\
 &=& \Big( \big( gh, \tim (H_g,H_h) \big) , \big( \bs(\phi_h)^c, \bs(\phi_h)^b \big)  \Big) \\
 &=&  F(\phi_g) F (\phi_h),
\end{eqnarray*}
while in the last line we used the multiplication of the action groupoid. Similarly, for the unit map one has
\begin{eqnarray*}
 F (\bu (\varphi^c_x, \varphi^b_x)) &=& \Big( \big(1_x, \Ima( \tiu (\varphi^b_x) ) \big), \bs (\bu (\varphi^c_x, \varphi^b_x))^c,  \bs (\bu (\varphi^c_x, \varphi^b_x))^b \Big) \\
 &=& \Big( \big(1_x, \tiu (E_M)_x \big), \varphi^c_x, \varphi^b_x \Big) \\
 &=&  u (\varphi^c_x, \varphi^b_x),
\end{eqnarray*}
and for the inversion
\begin{eqnarray*}
 F (\bi (\phi_g)) &=& \Big( \big( g^{-1}, \Ima (\tii (\phi_g (0, \cdot))) \big), \big( \bs (\bi (\phi_g))^c, \bs (\bi (\phi_g))^b \big) \Big) \\
 &=& \Big( \big( g^{-1}, \tii (\Ima (\phi_g (0, \cdot))) \big), \big( \bt (\phi_g)^c, \bt (\phi_g) \big) \Big) \\
 &=&  i (F(\phi_g)).
\end{eqnarray*}
%
%
%
\end{proof}

\begin{rk}[PB-groupoids and transitivity]
It is known that, given a VB-groupoid $E_\cG \tto E_M$ over $\cG \tto M$, if the groupoid $E_\cG \tto E_M$ is transitive, then $\cG \tto M$ is transitive as well (but the converse does not hold). For PB-groupoids the issue is more intricate, since there are more groupoid structures in play (including those of the structural Lie 2-groupoid). In particular, it is not clear if/when the groupoid structure on $\Fr(E_\cG)^{\rm \sbis} \tto \Fr(C) \times_M \Fr(E_M)$ from Theorem \ref{thm_special_frames_are_groupoid} is transitive; we are currently looking into it, with the goal of providing an equivalent description in term of a ``gauge PB-groupoid''. This could shed lights on the relation between our correspondence VB/PB-groupoids, and the correspondence DVB/DPB discussed in \cite{LangLiLiu21}.
\end{rk}


%

\subsection{The (set-theoretical) 2-action on the \frs\ frames}

As anticipated, after having shown that $\Fr^\sbis(E_\CG)\soutar \Fr(C)\times_M \Fr(E_M)$ is a set-theoretical groupoid (Theorem \ref{thm_special_frames_are_groupoid}), we describe now the (free) action of the general linear 2-groupoid $\GL(l,k)$ on $\Fr^\sbis (E_\cG)$. In the next section we will deal with the smoothness of these spaces, and such action will turn out to be smooth and principal.

To this end, let us  \cite{Sternberg64, Kobayashi95}recall (e.g.\ from  \cite{Sternberg64, Kobayashi95}) that, given any Lie subgroup $G \subseteq \GL(l+k)$, a \textbf{$G$-structure} on a vector bundle $E \to M$ of rank $l+k$ is a principal $G$-subbundle of the frame bundle $\Fr(E) \to M$. In particular, let $\GL^l(l+k)$ denote the group of block matrices $\begin{pmatrix}
A_1 & A_2\\
0 & A_4
\end{pmatrix}$, where $A_1 \in \GL(l)$ and $A_4 \in \GL(k)$. Then there is a 1-1 correspondence
\[
 \Bigg\{ \GL^l (l+k)\text{-structures on } E \Bigg\} \Longleftrightarrow \Bigg\{ \text{rank } l \text{ vector subbundle } D \subseteq E \text{ over } M \Bigg\} .
\]
Indeed, to any such $D$ one associates the subbundle $\Fr(E,D) \subseteq \Fr(E)$ consisting of frames $\phi_x \colon \RR^{l+k} \to E_x$ preserving $D_x$, i.e.\ such that $\phi_x (\RR^k \times \{0\}) \subseteq D_x$. The right $\GL(l+k)$-action on $\Fr(E)$, i.e.\ $(\phi_x, A) \mapsto \phi_x \circ A$, restricts to a $\GL^l(l+k)$-action on $\Fr(E,D)$, since $\phi_x (A (w,0)) = \phi_x (A_1 w,0)$. We conclude that $\Fr(E,D)$ is a principal $\GL^l(l+k)$-subbundle of $\Fr(E)$. Conversely, given any principal $\GL^l (l+k)$-subbundle $P \subseteq \Fr(E)$, one defines uniquely a rank $k$ subbundle $D \subseteq E$ by taking $D_x$ as the union of the subspaces $\phi_x (\RR^l \times \{0\}) \subseteq E_x$, for all $\phi_x \in P_x$, so that $P = \Fr(E,D)$.


\begin{thm}\label{prop_2-action_on_frames}
	Let $E_\CG\soutar E_M$ be a VB-groupoid of rank $(l,k)$ over a Lie groupoid $\CG\soutar M$. There the (principal) right Lie group actions of $\GL(l+k)$ on $\Fr(E_\cG)$ and of $\GL(l) \times \GL(k)$ on $\Fr(C) \times \Fr(E_M)$ induce a canonical right free 2-action of the (Lie) 2-groupoid $\GL(l,k)$ (Definition \ref{ex_GL_2_groupoid}) on the (set-theoretical) groupoid $\Fr^\sbis (E_\cG) \tto \Fr(C) \times \Fr(E_M)$ from Theorem \ref{thm_special_frames_are_groupoid}, with moment map given by Lemma \ref{lem_t_phi} as
	\[
	\mu \colon \Fr^\sbis(E_\CG) \to \GL(l,k)_0, \quad \phi_g \mapsto \dphig.
	\]
More precisely, the induced 2-action is defined as follows:
\begin{itemize}
 \item the right groupoid action of $\GL(l,k)_2 \tto \GL (l,k)_0$ on $\Fr^\sbis(E_\cG)$ is given by
 \[
  \phi_g \cdot \left(d,\left(\begin{matrix}
  	A & JB \\
  	0 & B
  \end{matrix}\right)\right) := \phi_g \cdot \left(\begin{matrix}
  A & JB \\
  0 & B
\end{matrix}\right);
 \]
\item the right groupoid action of $\GL(l,k)_1 \tto \GL (l,k)_0$ on $\Fr (C) \times_M \Fr(E_M)$ is given by
\[
 (\varphi_x^c,\varphi_x^b) \cdot (d,A,B) := (\varphi_x^c \cdot A, \varphi_x^b \cdot B).
\]
\end{itemize}
\end{thm}


%
%

\begin{proof}
First of all notice that $\ker(\tis) \subseteq E_\cG$ is a rank $l$ vector subbundle, and the corresponding $\GL^l(l+k)$-structure $\Fr(E_\cG,\ker(\tis))$ (see the discussion above) coincides precisely with the set  of all frames $\phi_g \in \Fr(E_\cG)$ satisfying the first condition of Definition \ref{def_special_frames}. Accordingly, the $\GL(l+k)$-action on $\Fr(E_\cG)$ restricts to a $\GL^l(l+k)$-action on $\Fr(E_\cG,\ker(\tis))$.


When imposing also the second condition of Definition \ref{def_special_frames}, i.e.\ restricting $\Fr(E_\cG,\ker(\tis))$ to $\Fr^\sbis (E_\cG)$, then the action of the Lie \textit{group} $\GL^l(l+k) \subseteq \GL(l+k)$ induces an action of the Lie \textit{2-groupoid} $\GL(l,k)$ from Definition-Example \ref{ex_GL_2_groupoid}. Indeed, by Lemma \ref{lem_t_phi}, the map $\mu$ satisfies $\mu(\phi_g)=\mu(\bu\bs\phi_g)=\mu(\bu\bt\phi_g)$ for any \frs\ frame $\phi_g$, hence this is a Lie 2-groupoid action as a direct consequence of Proposition \ref{prop.change.of.coordinates}. 
\end{proof}

\begin{lem}\label{prop.change.of.coordinates}
Let $E_\cG \tto E_M$ be a VB-groupoid of rank $(l,k)$ over $\cG \tto M$ and consider two composable arrows $g,h\in \CG$, three frames $\phi_g,\phi_g',\phi_g'' \in \Fr^{\sbis}(E_\CG)$, and two frames $\phi_{h}, {\phi_{h}'} \in \Fr^{\sbis}(E_\CG)$ such that $\phi_g,\phi_h$ and $\phi_g',\phi_h'$ are composable with respect to the groupoid structure on $\Fr^\sbis (E_\cG) \tto \Fr(C) \times_M \Fr(A)$ from Theorem \ref{thm_special_frames_are_groupoid}.
	\begin{enumerate}
		\item The unique matrix $M_g\in \GL(l+k)$ satisfying $\phi_g M_g= \phi_g'$ also satisfies $(d_{\phi_g},M_g)\in \GL(l,k)_2$ and $d_{\phi_g'}=\bs_{20}(d_{\phi_g},M_g)$.
		\item If $M_g',M_g''\in \GL(l+k)$ are the unique matrices satisfying $\phi_g' M_g'= \phi_g''$ and $\phi_g M_g''= \phi_g''$, then
		$$(d_{\phi_g},M_g'')=(d_{\phi_g},M_g)\circ_{20} (d_{\phi_g'},M_g').$$
		\item If $C,A\in\GL(l)$ and $D,B\in\GL(k)$ are the unique matrices satisfying 
		$$\bs(\phi_g)_c C= \bs(\phi_g')_c \,\,,\,\, \bs(\phi_g)_b B= \bs(\phi_g')_b\,\,,\,\, \bt(\phi_g)_c A= \bt(\phi_g')_c \y \bt(\phi_g)_b D= \bt(\phi_g')_b,$$
		then 
		$$(d_{\phi_g},C,B)=s_{21}(d_{\phi_g},M_g)\in \GL(l,k)_1 \y (d_{\phi_g},A,D)=t_{21}(d_{\phi_g},M_g)\in \GL(l,k)_1.$$
		\item If $M_h$ is the unique matrix satisfying $\phi_h M_h= \phi_h'$, and $M_{gh}$ is the unique matrix satisfying $(\phi_g\circ \phi_h) M_{gh}= (\phi_g'\circ \phi_h')$,
		then
		$$(d_{\phi_g},M_{gh})=(d_{\phi_g},M_{g})\circ_{21} (d_{\phi_h},M_{h}).$$
	\end{enumerate}
\end{lem}

\begin{proof}
	This is mostly a calculation. We will only show part 3; the other parts can be proved with similar arguments.
	
	Given any $\phi_g,\phi_g'\in \Fr(E_\CG)_g$, there exists unique matrices $A$, $B$, $C$, $D$ and $M_g$ (interpreted as ``changes of coordinates'') such that
	$$\begin{tikzcd}
		(\bt(\phi_g)^{c},\bt(\phi)^{b})\arrow[d, "A\times D"] & \phi_g \ar[l]\ar[r] \ar[d, "M_g"]	& (\bs(\phi_g)^{c},\bs(\phi_g)^{b}) \ar[d, "C\times B"]\\
		(\bt(\phi_g')^{c},\bt(\phi_g')^{b}) & \phi_g' \ar[l]\ar[r] & (\bs(\phi_g')^{c},\bs(\phi_g')^{b})
	\end{tikzcd}$$
 Because of the definitions of $\bs,\bt$ for the \frs\ frames, in particular the first $l$ coordinates in $\bt$ and the last $k$ coordinates on $\bt$,  the matrix $M_g\in \GL(l+k)$ must be of the form: 
	$$M_g= \left(\begin{matrix}
		A & JB \\
		0 & B
	\end{matrix}\right),$$
	for some $l\times k$ matrix $J$. Spelling out what this means one gets that:
	\begin{itemize}
		\item  $A$ is the unique matrix such that $\phi_g(Aw,0)-{\phi_g'}(w,0)=0$ for all $w\in \KR^l$;
		\item $B$ is the unique matrix such that $\tis\phi_g(0,Bv)-\tis{\phi_g'}(0,v)=0$  for all $v\in \KR^k$;
		\item  $J$ is the unique $l\times k$ matrix such that $\phi_g(JBv,Bv)={\phi_g'}(0,v)$ for every $v\in \KR^k$, equivalently $J=\phi_g^{-1}\left({\phi_g'}(0,B^{-1}\cdot)-\phi_g(0,\cdot)\right)$;
		\item  $\phi_g(Aw+JBv,Bv)=\phi_g'(w,v)$ for all $(w,v)\in \KR^{l+k}$.
	\end{itemize} 
	
	Now we want to calculate $D$ and $C$ in terms of $A,B,J$ and $d_{\phi_g}$. By definition, $D$ is the matrix such that  $\bt\phi_g(0,Dv)=\bt{\phi_g'}(0,v)$. Using the definition of $d_{\phi_g}$ for the $(I)$-equality and that $\phi_g M_g=\phi_g'$ in the $(II)$-equality we get: $$\bt\phi_g(0,d_{\phi_g} JBv+Bv)\stackrel{(I)}{=}\bt\phi_g(JBv,Bv)\stackrel{(II)}{=}\bt\phi_g'(0,v)=\bt\phi_g(0,Dv),$$ therefore we conclude $D=d_{\phi_g}JB+B=(I+d_{\phi_g}J)B.$
	
	The value for $D$ implies part of statement 1 of this Lemma, i.e. $\gds_{20}(d_{\phi_g}, M_g)=((I+d_{\phi_g}J)D)^{-1}d_{\phi_g}A=d_{\phi_g'}$. Later in this argument we will use this fact in the form of $ d_{\phi_g}A=(I+d_{\phi_g}J)Dd_{\phi_g'}$.

	Last, $C$ is the unique $l\times l$ matrix satisfying, for any $w\in \KR^l$, 
	$$\bs(\phi_g)^c(Cw)=\bs(\phi_g')^c(w) \,\,\, \Rightarrow \,\,\tiR_{g}\tii \big( \phi_g(Cw,-d_{\phi_g} Cw)\big)=\tiR_{g} \tii\big( {\phi_g'}(w,-d_{\phi_g'} w)\big),$$
	so that
	$$\phi_g(Cw,-d_{\phi_g} Cw)= {\phi_g'}(w,-d_{\phi_g'} w)=\phi_g(Aw-JDd_{\phi_g'} w,-Dd_{\phi_g'} w).$$
	Now we use that $d_{\phi_g}A=(I+d_{\phi_g}J)Dd_{\phi_g'}$ and obtain
	$$\phi_g(Cw,-d_{\phi_g} Cw)=\phi_g(Aw-J(I+d_{\phi_g}J)^{-1}d_{\phi_g} A w,-(I+d_{\phi_g}J)^{-1}d_{\phi_g} A w).$$
	Therefore we have $C=(I-J(I+d_{\phi_g}J)^{-1}d_{\phi_g})A$. To conclude the proof, we need to show that $I-J(I+d_{\phi_g}J)^{-1}d_{\phi_g}=(I+Jd_{\phi_g})^{-1}$. Indeed,
	\begin{eqnarray*}
		(I-J(I+d_{\phi_g}J)^{-1}d_{\phi_g})(I+Jd_{\phi_g})&=& I-J(I+d_{\phi_g}J)^{-1}d_{\phi_g}+Jd_{\phi_g}-J(d_{\phi_g}J+I)^{-1}d_{\phi_g}Jd_{\phi_g}\\
		&=& I-J\left((I+d_{\phi_g}J)^{-1}-I+(d_{\phi_g}J+I)^{-1}d_{\phi_g}J \right)d_{\phi_g}\\
		&=& I-J\left((d_{\phi_g}J+I)^{-1}(I+d_{\phi_g}J)-I\right)d_{\phi_g}=I,
	\end{eqnarray*}
	and similarly $(I+Jd_{\phi_g}) (I-J(I+d_{\phi_g}J)^{-1}d_{\phi_g}) = I$.\\
	
	The discussion above is summarised by the source and target as in as in Definition \ref{ex_GL_2_groupoid}:
	$$\begin{tikzcd}
		\left(d_{\phi_g},\begin{matrix}
			A \, , \, (I+d_{\phi_g} J)B=D
		\end{matrix}\right) & \left(d_{\phi_g},
		\left(\begin{matrix}
			A & JB \\
			0 & B
		\end{matrix}\right)\right) \ar[l,maps to, swap, "\gdt_{21}"] \ar[r,maps to,"\gds_{21}"]  &
		\left(d_{\phi_g},\begin{matrix}
        C=(I+Jd_{\phi_g})^{-1}A \, , \, B 
		\end{matrix}\right)
	\end{tikzcd}.$$
\end{proof}

\subsection{The smooth structure of the \frs\ frames}

As anticipated, in this section we take care of the smoothness issues, concluding the proof that $\Fr^\sbis (E_\cG)$ defines a PB-groupoid.

\begin{thm}\label{thm_special_frames_are_principal_bundles}
	Let $E_\CG\soutar E_M$ be a VB-groupoid of rank $(l,k)$ over a Lie groupoid $\CG\soutar M$. Then
	\begin{itemize}
		\item the total space of the set-theoretical groupoid $\Fr^\sbis(E_\CG)\soutar \Fr(C)\times_M \Fr(E_M)$ from Theorem \ref{thm_special_frames_are_groupoid} is a submanifold of $\Fr(E_\cG)$;
		\item the set-theoretical 2-action from Theorem \ref{prop_2-action_on_frames} is a smooth principal action.
	\end{itemize}
\end{thm}

\begin{proof}
We begin by proving the existence, around any point of $\cG$, of smooth local sections of $\Fr(E_\cG) \to \cG$ taking values in $\Fr^\sbis(E_\CG)$. First of all, since $\tis$ and $\tit$ are submersions, around any $g\in \CG$ there are local sections taking values in $\Fr(\ker(\tis))$ and in $\Fr(\ker(\tit))$. Let us think of frames as a local ordered basis: $(e_1,\dots,e_l)$ is an ordered basis for $\ker(\tis)$ and $(f_1,\dots,f_l)$ an ordered basis for $\ker(\tit)$. We now describe an algorithm to get \frs\ frames.
	
Set a counter $i=1$; take non-vanishing local sections $h_i$ not in the span of $(e_1,\dots,e_{l+i-1})$ and $h_i'$ not in the span of $(f_1,\dots,f_{l+i-1})$, then let $e_{l+i}=f_{l+i}:=h_i+ h_i'$ and redefine $i=i+1$. Repeating this process until $i=k+1$ we get that for $i$ between $l+1$ and $l+k$, the elements $e_i=f_i$ are linearly independent and not in $\ker(\tis)$ nor in $\ker(\tit)$, therefore the ordered basis of sections $(e_1,\dots,e_{l+k})$ is equivalent to an \frs\ frame.

Recall now that, since $\Fr(E_\cG) \to \cG$ is a principal bundle, the smooth structure of $\Fr(E_\cG)$ is given by smooth local sections $\phi_i:U_i\fto \Fr(E_\cG)$, which we can pack together as a section $\phi:\CU_\CG\fto \Fr(E_\cG)$, where $\CU_\CG:=\sqcup_{i\in I} U_i$ and $(U_i)_{i\in I}$ is an open cover of $\CG$. This structure is the only one making the local trivialisation
$$\varphi: \CU_\CG\times \GL(l+k)\fto \Fr(E_\cG), \quad \quad (g,A)\mapsto \phi(g)\cdot A$$
a local diffeomorphism. By the first part of the proof, we can assume that $\phi$ takes values in $\Fr^\sbis(E_\cG)$. Since being an immersed submanifold is a local property, we check that $\Fr^\sbis(E_\cG)$ is smooth using the local trivialisation $\varphi$. By part 1 of Lemma \ref{prop.change.of.coordinates}, the subset $\Fr^\sbis(E_\cG)\subset \Fr(E_\cG)$ locally looks like
$$\varphi^{-1}(\Fr^\sbis(E_\cG))= \CU_\CG \,{}_{d_\phi}\!\times_{\bt_{20}} \GL(l,k)_2.$$
Since $\bt_{20}$ is a submersion, and for any $g\in \CU_\CG$ the natural inclusion $\GL(l,k)_2|_{d_{\phi(g)}}\hookrightarrow \GL(l+k)$ is a submanifold, then $\CU_\CG \,{}_{d_\phi}\!\times_{\bt_{20}} \GL(l,k)_2\hookrightarrow \CU_\CG\times\GL(l+k)$ is a submanifold and so is $\Fr^\sbis(E_\cG)\subset \Fr(E_\cG)$.

Similarly, using part 3 of Lemma \ref{prop.change.of.coordinates} and the local trivialisation, one can check that the source and target maps of $\Fr^\sbis(E_\cG)\soutar \Fr(C)\times_M \Fr(E_M)$ are submersions, and using part 4 of the same lemma, that the multiplication is smooth. We conclude that $\Fr^\sbis(E_\cG)$ is a Lie groupoid.

Last, the 2-action is smooth since, by part 2 of Lemma \ref{prop.change.of.coordinates}, under the local trivialisation $\varphi$ it corresponds to the canonical (smooth) right action of $\GL(l,k)$ on $\CU_\CG \,{}_{d_\phi}\!\times_{\bt_{20}} \GL(l,k)_2$ with moment map
\[\mu\colon \CU_\CG \,{}_{d_\phi}\!\times_{\bt_{20}} \GL(l,k)_2\fto \GL(l,k)_0, \quad \quad (g,d_{\phi(g)},M_g)\mapsto \bs_{20}(d_{\phi(g)},M_g). \qedhere
 \]
\end{proof}


\subsection{Second half of the correspondence}

Now we move to the ``inverse'' of Theorem \ref{thm_from_PB_to_VB}. Recall that, given any vector bundle $E \to M$ of rank $k$, the set $\Fr(E)$ of its frames is a principal $\GL(k)$-bundle over $M$.

\begin{thm}\label{thm_from_VB_to_PB}
Given any VB-groupoid $E_\cG \tto E_M$ of rank $(l,k)$ over $\cG \tto M$, its \frs\ frame bundle $\Fr^\sbis(E_\CG)$
(Definition \ref{def_special_frames}) defines a PB-groupoid over $\cG \tto M$ (Definition \ref{def_PB_groupoid}), with  structural Lie 2-groupoid $\GL(l,k)$.
\end{thm}

\begin{proof}
It follows by combining Theorems \ref{thm_special_frames_are_groupoid}, \ref{prop_2-action_on_frames}, and \ref{thm_special_frames_are_principal_bundles}.
\end{proof}

\begin{thm}\label{thm_1-1_correspondence_VB_PB}
Let $\cG \tto M$ be a Lie groupoid; then Theorem \ref{thm_from_PB_to_VB} and Theorem \ref{thm_from_VB_to_PB} give a 1-1 correspondence between 
\begin{itemize}
 \item VB-groupoids of rank $(l,k)$ over $\cG \tto M$;
 \item PB-groupoids over $\cG \tto M$ with structural Lie 2-groupoid $\GL(l,k)$.
\end{itemize}
More precisely, any VB-groupoid $E_\cG \tto E_M$ is isomorphic to the VB-groupoid associated to the \frs\ frame bundle of $E_\cG$ and the canonical representation of $\GL(l,k)$ from Example \ref{ex_canonical_representation_of_GL}, i.e.\
\[\begin{tikzcd}
	\frac{\Fr^\sbis (E_\cG) \times_{\GL(l,k)_0} (\RR^l \times \RR^k \times \GL(l,k)_0)}{\GL(l,k)_2} & \cG \\
	\frac{(\Fr(C) \times_M \Fr(E_M)) \times_{\GL(l,k)_0} (\RR^k \times \GL(l,k)_0) }{\GL(l,k)_1} & M
	\arrow[from=1-1, to=2-1, shift left=.5ex, "\tis"]
	\arrow[from=1-1, to=2-1, shift right=.5ex, "\tit"']
	\arrow[from=1-2, to=2-2, shift left=.5ex, "s"]
	\arrow[from=1-2, to=2-2, shift right=.5ex, "t"']
	\arrow["{\pi_\cG}", from=1-1, to=1-2]
	\arrow["{\pi_M}", from=2-1, to=2-2]
\end{tikzcd}\]

Conversely, any PB-groupoid $P_\cG \tto P_M$ is isomorphic to the \frs\ frame bundle of its associated VB-groupoid (Definition \ref{def_special_frames}), i.e.\
\[\begin{tikzcd}
	\Fr^{\sbis} \left( \frac{P_\cG \times_{\GL(l,k)_0} (\RR^k \times \RR^l \times \GL(l,k)_0)}{\GL(l,k)_2} \right) & \cG \\
	\Fr \left( \frac{P_M \times_{\GL(l,k)_0} (\RR^k \times \GL(l,k)_0)}{\GL(l,k)_1} \right) \times_M \Fr \left( \frac{ P_{\cG}|_{M} \times_{\GL(l,k)_0} (\RR^l \times \GL(l,k)_0) }{\GL(l,k)_2} \right) & M
	\arrow[from=1-1, to=2-1, shift left=.5ex, "\bs"]
	\arrow[from=1-1, to=2-1, shift right=.5ex, "\bt"']
	\arrow[from=1-2, to=2-2, shift left=.5ex, "s"]
	\arrow[from=1-2, to=2-2, shift right=.5ex, "t"']
	\arrow["{\Pi_\cG}", from=1-1, to=1-2]
	\arrow["{\Pi_M}", from=2-1, to=2-2]
\end{tikzcd}\]
\end{thm}

\begin{proof}

For the first part, it is enough to check that 
 \[
 \frac{\Fr^\sbis (E_\cG) \times_{\GL(l,k)_0} (\RR^l \times \RR^k \times \GL(l,k)_0)}{\GL(l,k)_2} \to E_\cG, \quad [\phi, (v,w,g)] \mapsto \phi(v,w)
  \]
and
 \[
 \frac{(\Fr(E_M) \times \Fr(C)) \times_{\GL(l,k)_0} (\RR^k \times \GL(l,k)_0) }{\GL(l,k)_1} \to E_M, \quad [({\varphi^b_x},{\varphi^c_x}), (v,g)] \mapsto \varphi^c_x(v)
  \]
are well defined isomorphisms of vector bundles.

For the second part, it is enough to check that, for every $p \in P_{\cG}$, the map
 \[
  \phi_p: \RR^k \times \RR^l \times \GL(l,k)_0 \to \frac{P_\cG \times_{\GL(l,k)_0} (\RR^k \times \RR^l \times \GL(l,k)_0)}{\GL(l,k)_2}, \quad (v,w,g) \mapsto [p,(v,w,g)]
 \]
is an \frs\ frame, and therefore
 \[
 P_{\cG} \to \Fr^{\sbis} \left( \frac{P_\cG \times_{\GL(l,k)_0} (\RR^k \times \RR^l \times \GL(l,k)_0)}{\GL(l,k)_2} \right), \quad p \mapsto \phi_p
  \]
is an isomorphism of principal $(\GL(l,k)_2 \tto \GL(l,k)_0)$-bundles. Similarly,
\[
P_M \to \Fr \left( \frac{P_M \times_{\GL(l,k)_0} (\RR^k \times \GL(l,k)_0)}{\GL(l,k)_1} \right) \times_M \Fr \left( \frac{ P_{\cG}|_{M} \times_{\GL(l,k)_0} (\RR^l \times \GL(l,k)_0) }{\GL(l,k)_2} \right), \quad q \mapsto \bs (\phi_{\bu(q)}) = \bt (\phi_{\bu(q)})
\]
is an isomorphism of principal  $(\GL(l,k)_1 \tto \GL(l,k)_0)$-bundles.
\end{proof}

\section{Examples}\label{sec_examples}

We illustrate the correspondence between VB- and PB-groupoids by describing the \frs\ frame bundle of several classes of VB-groupoids.

\subsection{Special cases of VB-groupoids}

\begin{ex}[trivial core]\label{ex_rank_0k}
Let $E_\cG \tto E_M$ be a VB-groupoid of rank $(0,k)$, i.e.\ with core $C = 0$ and $\rank(E_\cG \to \cG) = \rank (E_M \to M) = k$ (Example \ref{ex_VB_groupoid_trivial_core}). Then its \frs\ frame bundle boils down to the Lie groupoid
\[
\Fr(E_\cG) \tto \Fr(E_M)
\]
with structure maps
\[
 \bs \colon \Fr(E_\cG) \to \Fr(E_M), \quad \bs (\phi_g) = \tis \circ \phi_g;
\]
\[
 \bt \colon \Fr(E_\cG) \to \Fr(E_M), \quad \bt (\phi_g) = \tit \circ \phi_g;
\]
\[
 \bm \colon \Fr(E_\cG) \tensor[_{\bs}]{\times}{_{\bt}} \Fr(E_\cG)\to \Fr(E_\cG), \quad \bm (\phi_g, \phi_h ) (v) := \tim \Big( \phi_g (v), \phi_h (v) \Big);
\]
\[
 \bu \colon \Fr(E_M) \to \Fr(E_\cG), \quad \bu (\varphi_x) := \tiu \circ \varphi_x;
\]
\[
 \bi \colon \Fr(E_\cG) \to \Fr(E_\cG), \quad \bi (\phi_g) := \tii \circ \phi_g. 
\]

Indeed, given a frame $\phi_g \colon \RR^{0+k} \to (E_\cG)_g$, the first condition in Definition \ref{def_special_frames} is trivially satisfied since $\phi_g|_{\RR^0 \times \{0\}} = 0$, while the second condition holds true as $\phi_g|_{\{0\} \times \RR^k} = \phi_g$, hence its image coincides with the entire $(E_\cG)_g$ since $\phi_g$ is a linear isomorphism.

{
Equivalently, since the fat groupoid $\tilde{\cG}(E_\cG)$ is isomorphic to $\cG$ (Example \ref{ex_fat_groupoid_trivial_base_core}), the action groupoid $\tilde{\cG}(E_\cG) \times_M ( \Fr(C) \times_M \Fr(E_M))$ becomes simply $\cG \times_M \Fr(E_M)$. In turn, this is an equivalent description of $\Fr(E_\cG)$ since $E_\cG \cong \cG \times_M E_M$.
}

Note also that, due to $\Hom (\RR^0, \RR^k) = \{0\}$, the structural Lie 2-groupoid $\GL(0,k)$ boils down to a Lie 2-group (Definition \ref{def:lie2grp}), simply given by the unit groupoid $\GL(k) \tto \GL(k)$.
\end{ex}

\begin{ex}[trivial base]\label{ex_rank_l0}
 Let $E_\cG \tto E_M$ be a VB-groupoid of rank $(l,0)$, i.e.\ with $E_M = 0$ and $\rank(E_\cG \to \cG) = \rank (C \to M) = l$ (Example \ref{ex_VB_groupoid_trivial_base}). Then its \frs\ frame bundle boils down to the Lie groupoid
 \[
  \Fr(E_\cG) \tto \Fr(C)
 \]
with structure maps
\[
 \bs \colon \Fr(E_\cG) \to \Fr(C), \quad \bs (\phi_g) := - \tiR_g \circ \tii \circ \phi_g;
\]
\[
 \bt \colon \Fr(E_\cG) \to \Fr(C), \quad \bt (\phi_g) = \tiR_{g^{-1}} \circ \phi_g;
\]
\[
 \bm \colon \Fr(E_\cG) \tensor[_{\bs}]{\times}{_{\bt}} \Fr(E_\cG) \to \Fr(E_\cG), \quad \bm (\phi_g, \phi_h ) (w) := \tiR_h ( \phi_g (w));
\]
\[
 \bu \colon \Fr(C) \to \Fr(E_\cG), \quad \bu (\varphi_x) := \varphi_x;
\]
\[
 \bi \colon \Fr(E_\cG) \to \Fr(E_\cG), \quad \bi (\phi_g) := - \tii \circ \phi_g.
\]

Indeed, given a frame $\phi_g \colon \RR^{l+0} \to (E_\cG)_g$, both conditions in Definition \ref{def_special_frames} are trivially satisfied since $\tis$ and $\tit$ are the zero map, hence their kernels coincide with the entire $E_\cG$.

{
Equivalently, since the fat groupoid $\tilde{\cG}(E_\cG)$ is isomorphic to $\cG$ (Example \ref{ex_fat_groupoid_trivial_base_core}), the action groupoid $\tilde{\cG}(E_\cG) \times_M ( \Fr(C) \times_M \Fr(E_M))$ becomes simply $\cG \times_M \Fr(C)$. In turn, this is an equivalent description of $\Fr(E_\cG)$ since $E_\cG \cong \cG \times_M C$.
}

Note also that, due to $\Hom (\RR^l, \RR^0) = \{0\}$, the structural Lie 2-groupoid $\GL(l,0)$ boils down to a Lie 2-group (Definition \ref{def:lie2grp}), simply given by the unit groupoid $\GL(l) \tto \GL(l)$.
\end{ex}

\begin{ex}[pullback VB-groupoid]
Let $\cG \tto M$ be a Lie groupoid and $\pi: E_M \to M$ be a vector bundle of rank $k$, and consider the pullback VB-groupoid $E_\cG := E_M \tensor[_{\pi}]{\times}{_{t}} \cG \tensor[_{s}]{\times}{_{\pi}} E_M \to \cG$, which has rank $(k,k)$ (Example \ref{ex_pullback_VB_groupoid}). Then its \frs\ frame is isomorphic to a Lie subgroupoid of the pullback groupoid $P_\cG$ of $\cG$ with respect to the principal bundle $\Fr(E_M) \times_M \Fr(E_M) \to M$, i.e.\
\[
P_\cG := \Big( \Fr(E_M) \times_M \Fr(E_M) \Big) \tensor[_{\pi}]{\times}{_{t}} \cG \tensor[_{s}]{\times}{_{\pi}} \Big( \Fr(E_M) \times_M \Fr(E_M) \Big) \tto \Fr(E_M) \times_M \Fr(E_M);
\]
more precisely,
\[
\Fr^{\sbis} (E_\cG) \cong \Big\{ \Big(\varphi^1_{t(g)}, \varphi^2_{t(g)}, g, \varphi^3_{s(g)}, \varphi^4_{s(g)} \Big) \in P_\cG | \varphi^3_{s(g)} = \varphi^4_{s(g)} \circ (\varphi^2_{t(g)})^{-1} \circ \varphi^1_{t(g)} \Big\} \subseteq P_\cG.
\]

Indeed, any frame $\phi_g$ of $E_\cG \to \cG$ at $g \in \cG$ takes the form
\[
 \phi_g = (\phi_g^1,g,\phi_g^3): \RR^{2k} \to (E_\cG)_g = (E_M)_{t(g)} \times \{g\} \times (E_M)_{s(g)}.
\]
Then the first condition from Definition \ref{def_special_frames} becomes
\[
 \phi_g^3 (w,0) = 0,
\]
which implies that $\phi^1_g (\cdot,0)$ is a frame of $E_M \to M$, while the second condition guarantees that $\phi^1_g (0,\cdot)$ and $\phi^3_g (0,\cdot)$ are frames of $E_M \to M$. Moreover, since the core-anchor $\rho: E_M \to E_M$ is the identity, the map $d_{\phi}: \RR^k \to \RR^k$ from Lemma \ref{lem_t_phi} is defined by
\[
 \phi^1_g (w,0) = \phi^1_g (0,d_{\phi}(w)) \quad \forall w \in \RR^k,
\]
i.e.\ it is the linear isomorphism
\[
 d_\phi = (\phi^1_g (0,\cdot))^{-1} \circ \phi_g^1 (\cdot,0).
\]
This shows that any \frs\ frame $\phi_g$ is encoded in the arrow $g$ together with a quadruple of frames of $E_M$, namely
\[
 \Big( \varphi^1_{t(g)}:= \phi^1_g (\cdot,0), \varphi^2_{t(g)}:= \phi^1_g (0,\cdot), \varphi^3_{s(g)}:= \phi^3_g (0, d (\cdot)), \varphi^4_{s(g)}:= \phi^3_g (0,\cdot) \Big),
\]
where $\varphi^3_{s(g)}$ is by construction the composition $\varphi^4_{s(g)} \circ d_\phi = \varphi^4_{s(g)} \circ (\varphi^2_{t(g)})^{-1} \circ \varphi^1_{t(g)}$. Moreover, the restrictions of the structure maps of the pullback groupoid $P_\cG \tto \Fr(E_M) \times_M \Fr(E_M)$ coincide with those of Theorem \ref{thm_special_frames_are_groupoid}.
\end{ex}

\begin{ex}[canonical VB-groupoid]
The fact that the frame bundle of the canonical vector space $\RR^n$ coincides with general linear group $\GL(n)$, with the principal action of $\GL(n)$ on itself, can be generalised to VB-groupoids. More precisely, we will show that the \frs\ frame bundle of the canonical VB-groupoid $\RR^{(l,k)}$ (Example \ref{ex_trivial_VB_groupoid}) coincides with the PB-groupoid associated, via Example \ref{ex_lie_2_groupoids_as_PB_groupoids}, to the general linear 2-groupoid $\GL(l,k)$ (Example \ref{ex_GL_2_groupoid}), i.e.\ with
\[\begin{tikzcd}
{\GL(l,k)_2} & \GL(l,k)_0 \\
{\GL(l,k)_1} & \GL(l,k)_0
\arrow[from=1-1, to=2-1, shift left=.5ex, "s_{21}"]
\arrow[from=1-1, to=2-1, shift right=.5ex, "t_{21}"']
\arrow[from=1-2, to=2-2, shift left=.5ex, "\id"]
\arrow[from=1-2, to=2-2, shift right=.5ex, "\id"']
\arrow["{t_{20}}", from=1-1, to=1-2]
\arrow["{t_{10}}", from=2-1, to=2-2]
\end{tikzcd}\]
with the principal action of $\GL(l,k)$ on itself.

Indeed, one checks immediately that the frame bundle of the core of $\RR^{(l,k)}$ is 
\[
 \Fr(C) = \Fr \Big(\RR^l \times \GL(l,k)_0 \to \GL(l,k)_0 \Big) \cong \GL(l) \times \GL(l,k)_0,
\]
while the frame bundle of the base of $\RR^{(l,k)}$ is
\[
 \Fr(E_M) = \Fr \Big( \RR^k \times \GL(l,k)_0 \to \GL(l,k)_0 \Big) \cong \GL(k) \times \GL(l,k)_0,
\]
so that their fibred product is
\[
   \Fr(C) \times_{\GL(l,k)_0} \Fr(E_M) \cong \GL(l,k)_0 \times \GL(l) \times \GL(k) = \GL(l,k)_1.
\]

On the other hand, the frame bundle of the trivial vector bundle $\RR^{(l,k)}_2 = \RR^l \times \RR^k \times \GL(l,k)_0 \to \GL(l,k)_0$ is the trivial bundle principal bundle
\[
 \GL(l+k) \times \GL(l,k)_0 \to \GL(l,k)_0.
\]
To conclude that $\Fr^\sbis (\RR_2^{(l,k)}) = \GL(l,k)_2$, it is enough to check that the conditions of Definition \ref{def_special_frames} coincide with those defining $\GL(l,k)_2 \subseteq \GL(l+k) \times \GL(l,k)_0$ in Example \ref{ex_GL_2_groupoid}. To this end, consider an element $g \in \cG = \GL(l,k)_0$ and a frame $\phi_g \colon \RR^l \times \RR^k \to (\RR_2^{(l,k)})_g = \RR^l \times \RR^k \times \{g\}$, which is identified with a block matrix in $\GL(l+k)$, i.e.\
\[
 \phi_g (w,v) = \left(\begin{matrix}
			M_1 & M_2 \\
			M_3 & M_4
		\end{matrix}\right) \left(\begin{matrix}
			w \\
			v
		\end{matrix}\right) = 
		\left(\begin{matrix}
			M_1 w + M_2 v \\
			M_3 w+  M_4 v
		\end{matrix}\right).
\]
Then condition 1 of Definition \ref{def_special_frames} is equivalent to $M_3 = 0$, since $\ker(\tis)_g = \RR^l \times \{0\} \times \{g\}$. Similarly, from condition 2 of Definition \ref{def_special_frames} it follows that $M_4$ is invertible (hence $M_1$ is invertible as well, by the properties of block matrices). Setting $J := M_2 (M_4)^{-1} \colon \RR^k \to \RR^l$, one can check the rest of the conditions defining $\GL(l,k)_2$ via computations completely analogous to those in part 3 of Lemma \ref{prop.change.of.coordinates}.
\end{ex}

\subsection{Tangent VB-groupoids}

\begin{ex}[tangent VB-groupoids]\label{ex_PB_of_tangent_groupoid}
Let $\cG \tto M$ be a Lie groupoid and consider its tangent VB-groupoid $E_\cG = T\cG \tto E_M = TM$, whose core $C$ coincides with the Lie algebroid $A = \mathrm{Lie}(\cG)$ (Example \ref{ex_tangent_VB_groupoid}). Then the Lie groupoid structure of the \frs\ frame bundle $\Fr^\sbis (\cG) \tto \Fr(A) \times_M \Fr(M)$ admits a simpler description as
\[
 \bs \colon \Fr(\cG)^{\rm \sbis} \to \Fr(A) \times_M \Fr(M), \quad \bs (\phi_g) (w,v) = \Big( d_{g^{-1}}R_g ( d\tau (\phi_g (-w,\dphig( w ))) )
 , d_gs (\phi_g (0,v) ) \Big);
\]
\[
 \bt \colon \Fr(\cG)^{\rm \sbis} \to \Fr(A) \times_M \Fr(M), \quad \bt (\phi_g) (w,v) = \Big(d_g R_{g^{-1}} ( \phi_g (w,0) ), d_gt (\phi_g (0,v) ) \Big);
\]
\[
 \bm \colon \Fr(\cG)^{\rm \sbis} \tensor[_{\bs}]{\times}{_{\bt}} \Fr(\cG)^{\rm \sbis}\to \Fr(\cG)^{\rm \sbis}, \quad \bm (\phi_g, \phi_h ) (w,v) := d_{(g,h)}m \Big( \phi_g (0,v), \phi_h (0,v) \Big) + d_g R_h ( \phi_g (w,0));
\]
\[
 \bu \colon \Fr(A) \times_M \Fr(M) \to \Fr(\cG)^{\rm \sbis}, \quad \bu (\varphi_x^a, \varphi^b_x) (w,v) := d_x u (\varphi^b_x ( v ) ) + \varphi^a_x (w);
\]
\[
 \bi \colon \Fr(\cG)^{\rm \sbis} \to \Fr(\cG)^{\rm \sbis}, \quad \bi (\phi_g) (w,v):= d_g \tau (\phi_g (-w,v + \dphig(w)) ).
\]

{
Equivalently, since $\tilde{\cG}(T\cG)$ coincides with the jet groupoid $J^1 \cG$, i.e.\ the space of 1-jets of bisections of $\cG$, one can interpret $\Fr(\cG)^{\rm \sbis}$ as the action groupoid $J^1\cG \times_M (\Fr(A) \times_M \Fr(M))$, for the canonical representations of $J^1\cG$ on $A$ and $TM$.
}

One can also write down the structure maps of the \frs\ frame bundle of the cotangent VB-groupoid $T^* \cG$ (Example \ref{ex_cotangent_VB_groupoid}) and describe a canonical isomorphism of PB-groupoids $\Fr^\sbis (T\cG) \to \Fr^\sbis (T^*\cG)$ over $\Fr(A) \times_M \Fr(TM) \to \Fr(T^*M) \times_M \Fr(A^*)$. This will be discussed in greater generality in the next section.
\end{ex}

We consider now a few special cases of Example \ref{ex_PB_of_tangent_groupoid}.

\begin{ex}[pair groupoids]
Let $\cG = M \times M \tto M$ be the pair groupoid of an $n$-dimensional manifold $M$, and consider its tangent VB-groupoid $T\cG = TM \times TM \tto TM$, which is the pair groupoid of $TM$ and has rank $(n,n)$. A frame $\phi_{(x,y)}$ of $T\cG$ at $g = (x,y) \in M \times M$ takes the form
\[
 \phi_{(x,y)} \colon \RR^{2n} \to T_x M \times T_y M, \quad (w,v) \mapsto (\phi_x (w,v), \phi_y (w,v) ).
\]
Then the first condition from Definition \ref{def_special_frames} becomes
\[
\phi_y (w,0) = 0 \quad \text{for every } w \in \RR^n,
\]
which implies that $\phi_y (0,\cdot)$ is a frame of $M$, while the second condition guarantees that $\phi_x (\cdot, 0)$ and $\phi_x (0,\cdot)$ are frames of $M$. Last, since the core-anchor of $TM \times TM \tto TM$ is the anchor of the Lie algebroid $TM \to M$, i.e.\ the identity $\id_{TM}$, the map $\dphi \colon \RR^n \to \RR^n$ from Lemma \ref{lem_t_phi} is a linear isomorphism defined by
\[
 \phi_x (w,0) = \phi_x (0, \dphi(w)).
\]
Using the definition of the structure maps of the pair groupoid, together with the fact that
\[
 d_{(x,y)} R_{(y,z)} (\alpha_x, 0_y) = (\alpha_x, 0_z),
\]
one can see that the groupoid structure of $\Fr(M \times M)^{\rm \sbis} \tto  \Fr(M) \times_M \Fr(M)$ is given by
\[
\bs (\phi_{(x,y)}) = \Big( \phi_y (0, \dphig( \cdot )), \phi_y (0,\cdot)  \Big), \quad \quad \bt (\phi_{(x,y)}) = \Big( \phi_x (\cdot,0), \phi_x (0,\cdot) \Big);
\]
\[
 \bm (\phi^1_{(x,y)}, \phi^2_{(y,z)} ) (w,v) := \Big( \phi_x^1 (0,v), \phi_z^2 (0,v) \Big) + \Big( \phi_x^1 (w,0), 0_z \Big);
\]
\[
\bu (\varphi_x^a, \varphi^b_x) (w,v) := (\varphi^b_x ( v ), \varphi^b_x ( v ) ) + (\varphi^a_x (w), 0_x), \quad \quad \bi (\phi_g) (w,v):= \Big( \phi_y (0, v + \dphig(w)), \phi_x (0, v) \Big). \qedhere
\]
\end{ex}

\begin{ex}[\'etale Lie groupoids]
Let $\cG \tto M$ be an \'etale Lie groupoid, with $\dim(\cG) = \dim(M) = k$, and consider its tangent VB-groupoid $E_\cG = T\cG \tto E_M = TM$, which has rank $(0,k)$ and core $\mathrm{Lie}(\cG) = 0$. As per Example \ref{ex_rank_0k}, both conditions in Definition \ref{def_special_frames} are empty, therefore
\[
 \Fr(\cG)^{\rm \sbis} = \Fr(\cG),
\]
and the groupoid structure of $\Fr(\cG) \tto \Fr(M)$ is given by
\[
 \bs (\phi_g) = d_g s \circ \phi_g, \quad \quad \bt (\phi_g) = d_g t \circ \phi_g;
\]
\[
 \bm (\phi_g, \phi_h ) (v) := d_{(g,h)} m \Big( \phi_g (v), \phi_h (v) \Big);
\]
\[
 \bu (\varphi_x) := d_xu \circ \varphi_x, \quad \quad \bi (\phi_g) := d_g \tau \circ \phi_g. 
\]
A particular instance of this example is given by the unit groupoid $\cG = M \tto M$: since $s = t = \id_M$, the formulae above show that its \frs\ frame bundle is precisely the unit groupoid of $\Fr(M)$, i.e.\ its multiplication reduces to
\[
  \bm (\phi_g, \phi_h ) = \phi_g = \phi_h. \qedhere
\]
\end{ex}

\begin{ex}[Lie groups]
Let $G \tto M=\{*\}$ be an $l$-dimensional Lie group and consider its tangent VB-groupoid $E_{\cG} = TG \tto E_M = \{*\}$, which has rank $(l,0)$ and core $\g = \mathrm{Lie}(G)$. As per Example \ref{ex_rank_l0}, both conditions in Definition \ref{def_special_frames} are empty, therefore
\[
 \Fr(G)^{\rm \sbis} = \Fr(G),
\]
and the groupoid structure of $\Fr(G) \tto \Fr(\g)$ is given by
\[
 \bs (\phi_g) := - d_{g^{-1}}R_g \circ d_g \tau \circ \phi_g = d_g L_{g^{-1}} \circ \phi_g, \quad \quad \bt (\phi_g) = d_g R_{g^{-1}} \circ \phi_g;
\]
\[
 \bm (\phi_g, \phi_h ) (w) := d_g R_h ( \phi_g (w));
\]
\[
\bu (\varphi_x) := \varphi_x, \quad \quad \bi (\phi_g) := - d_g \tau \circ \phi_g. \qedhere
\]
\end{ex}

\begin{ex}[bundles of Lie groups]
Let $\pi: \cG \to M$ be a bundle of Lie groups, i.e.\ a Lie groupoid where source and target coincide, and consider its tangent VB-groupoid $d\pi: T\cG \to TM$, which is a bundle of Lie groups as well. Then
\[
 \Fr^\sbis (\cG) = \{ \phi_g \in \Fr(\cG) | \phi_g (w,0) \in \ker(d_g \pi) \ \forall w \in \RR^l \},
\]
since the first condition in Definition \ref{def_special_frames} implies the second one by Lemma \ref{lemma_automatic_s_section} and using $s = t = \pi$.

Moreover, the map $\dphi$ from Lemma \ref{lem_t_phi} is zero. Indeed, the element $\dphi (w) \in \RR^k$ is defined by
\[
d\pi (\phi_g (0,\dphi(w)) ) = d\pi (\phi_g (w, 0) ) = 0,
\]
i.e.\ by $\phi_g (0,\dphi(w)) \in \ker(d\pi)$. However, $\phi_g (0,\dphi(w))$ belongs also to $\Ima(\phi_g |_{\{0\} \times \RR^k})$, therefore it vanishes due to Lemma \ref{lemma_equivalent_def_bisection_frames}. Since $\phi_g$ is an isomorphism, it follows that $\dphi(w) = 0$. Then the groupoid structure of $\Fr(\cG)^{\rm \sbis} \tto \Fr(A) \times \Fr(M)$ can be written as 
\[
\bs (\phi_g) (w,v) = \Big( d_{g^{-1}}R_g ( d\tau (\phi_g (-w,0)) )
 , d_g\pi (\phi_g (0,v) ) \Big);
\]
\[
\bt (\phi_g) (w,v) = \Big(  d_gR_{g^{-1}} ( \phi_g (w,0) ),  d_g\pi (\phi_g (0,v) ) \Big);
\]
\[
 \bm (\phi_g, \phi_h ) (w,v) := d_{(g,h)}m \Big( \phi_g (0,v), \phi_h (0,v) \Big) + d_gR_h ( \phi_g (w,0));
\]
\[
\bu (\varphi_x^a, \varphi^b_x) (w,v) := d_x u (\varphi^b_x ( v ) ) + \varphi^a_x (w), \quad \quad \bi (\phi_g) (w,v):= d_g \tau (\phi_g (-w,v) ). \qedhere
\]
\end{ex}

\subsection{PB-groupoid of the dual VB-groupoid}


Recall that the fat groupoid of any VB-groupoid $E_\cG \tto E_M$ is canonically isomorphic to the fat groupoid of its dual VB-groupoid $E^*_\cG \tto C^*$ (Example \ref{ex_dual_VB_groupoid}); the isomorphism is given by
\[
\widetilde{\cG} (E_\cG) \to \widetilde{\cG} (E^*_\cG), \quad (g,H_g) \to (g,H^\circ_g),
\]
where $H^\circ_g \subseteq (E^*_\cG)_g$ is the annihilator of $H_g \subseteq (E_\cG)_g$. Then one has the following chain of isomorphisms of Lie groupoids:
\[
 \Fr^\sbis(E^*_\cG) \cong \widetilde{\cG} (E^*_\cG) \times_M (\Fr(E^*_M) \times_M \Fr(C^*)) \cong \widetilde{\cG} (E_\cG) \times_M (\Fr(C) \times_M \Fr(E_M)) \cong \Fr^\sbis (E_\cG).
\]
The goal of this subsection is to describe the isomorphism $\Fr^\sbis(E^*_\cG) \cong \Fr^\sbis (E_\cG)$ as a PB-groupoid isomorpism and provide an explicit formula (see Propositions \ref{prop_isomorphism_s_t_bisection_frames} and \ref{prop_PB_groupoid_dual_VB_groupoid} below) starting from the definition of \frs\ frame bundle of a VB-groupoid.
To this end, as anticipated, we will have to consider a version of Definition \ref{def_special_frames} obtained by switching $\tis$ and $\tit$ and replacing $\RR^l \times \RR^k$ with $\RR^k \times \RR^l$. This will yield a PB-groupoid structure which is isomorphic to the one described in Theorems \ref{thm_special_frames_are_groupoid} and
\ref{thm_special_frames_are_principal_bundles}, and which will work as a bridge between the \frs\ and the \frt\ frame bundles.

\begin{defn}\label{def_special_frames_t}
 Let $E_\cG \tto E_M$ be a VB-groupoid of rank $(l,k)$ over $\cG \tto M$. A frame $\phi_g \in \Fr(E_{\cG})$ is called a {\bf \frt\ frame} if the following two conditions holds:
 \begin{enumerate}
 \item $\phi_g|_{\{0\} \times \RR^l}$ takes values in $\ker(\tit_g)$;
 \item the image $H_g \subseteq (E_\cG)_g$ of $\phi_g|_{\{0\} \times \RR^k}$ is an element of the fat groupoid $\tilde{\cG}(E_\cG)$ (Definition \ref{def_fat_groupoid}).
 \end{enumerate}
 The collection
 \[
 \Fr(E_\cG)^\tbis := \{ \phi_g \colon \RR^{l+k} \to (E_\cG)_g \mid \phi_g \text{ \frt\ frame} \} \subseteq \Fr(E_\cG).
\]
of \frt\ frames of $E_\cG$ is called the \textbf{\frt\ frame bundle} of $E_\cG$.
 \end{defn}

\begin{prop}\label{prop_t_bis_frame_is_PB_groupoid}
 Given a VB-groupoid $E_\cG \tto E_M$ of rank $(l,k)$ over a Lie groupoid $\CG\soutar M$, its \frt\ frame bundle $\Fr(E_\cG)^\tbis$ (Definition \ref{def_special_frames_t}) is a Lie groupoid over $\Fr(E_M) \times_M \Fr(C)$, where $C \to M$ is the core of $E_\cG \tto E_M$ (Definition \ref{def_core}), with structure maps induced by those of $E_\cG \tto E_M$.
 
 Moreover, $\Fr(E_\cG)^\tbis$ is a (\textit{left}) PB-groupoid over $\cG \tto M$ with structural Lie 2-groupoid $\GL(l,k)$ (Example \ref{ex_GL_2_groupoid}), with principal left 2-action given by
 \[
  A \cdot \phi_g := \phi_g \circ T \circ A^T \circ T,
 \]
where $T$ is the map which flips the first and second group of coordinates in $\RR^{l+k}$:
\[
 T: \RR^l \times \RR^k \to \RR^k \times \RR^l, \quad (w_1,\ldots,w_l,v_1,\ldots,v_k) \mapsto (v_k,\ldots,v_1, w_l,\ldots,w_1).
\]
\end{prop}

\begin{proof}
The proof is formally identical to those of Theorems \ref{thm_special_frames_are_groupoid} and
\ref{thm_special_frames_are_principal_bundles}. In particular, note that, if $A \in \GL(l,k)_2$, then $T \circ A^T \circ T$ sends $\{0\} \times \RR^l$ to $\{0\} \times \RR^l$ and $\RR^k \times \{0\}$ to $\RR^k \times \{0\}$, hence it preserves the definition of \frt\ frame. The left action is therefore well-defined since $(AB)^T = B^T A^T$ and $T$ is an involution.
\end{proof}

Given any VB-groupoid, the PB-groupoid structure of its \frs\ frame bundle and that of its \frt\ frame bundle are actually the same, up to the inversion $\tii$ of $E_\cG$ and the flip $T$.

\begin{prop}\label{prop_isomorphism_s_t_bisection_frames}
 Let $E_\cG \tto E_M$ be a VB-groupoid of rank $(l,k)$ over a Lie groupoid $\CG\soutar M$. Then there is a canonical PB-groupoid isomorphism
  \[
 \psi \colon \Fr^\sbis (E_\cG) \to \Fr^\tbis (E_\cG), \quad \phi_g \mapsto \tii \circ \phi_g \circ T,
 \]
over the diffeomorphism
\[
 \Fr(C) \times_M \Fr(E_M) \to \Fr(E_M) \times_M \Fr(C), \quad (\varphi^c,\varphi^b) \mapsto (\varphi^b,\varphi^c),
\]
and with respect to the Lie 2-groupoid isomorphism
\[
 \GL(l,k) \to \GL(l,k), \quad A \mapsto A^T.
\]
\end{prop}

\begin{proof}
The flip $T$ guarantees that the restrictions of a frame $\phi_g$ in Definition \ref{def_special_frames} are mapped precisely to the restrictions in Definition \ref{def_special_frames_t}, while the inversion $\tii$ transforms $\ker(\tis)$ into $\ker(\tit)$, so that \frs\ frames are sent to \frt\ frames. This defines an isomorphism of Lie groupoids, with the inverse given by the same map $\phi_g \mapsto \tii \circ \phi_g \circ T$, since both $\tii$ and $T$ are involutions.

Moreover, $\psi$ is equivariant with respect to the right and the left principal 2- action of $\GL(l,k)$ and the isomorphism $A \mapsto A^T$, since
\begin{eqnarray*}
 A^T \cdot \psi(\phi) &=& \psi(\phi) \circ T \circ (A^T)^T \circ T \\
 &=& \tii \circ \phi \circ T \circ T \circ A \circ T \\
 &=& \tii \circ \phi \circ A \circ T \\
 &=& \tii \circ (\phi \cdot A) \circ T \\
 &=& \psi (\phi \cdot A).
\end{eqnarray*}
%
\end{proof}

Recall that, given a frame $\phi\colon\KR^{l+k}\fto V$ on a vector space $V$, its {\bf dual frame} is the isomorphism $\phi^*\colon\KR^{l+k}\fto V^*$  satisfying
$$\phi^*{(e_i)}(\phi(e_j))=\delta_{ij},$$
where $\{e_i\}$ is the canonical basis of $\KR^{k+l}$. Accordingly, for every vector bundle $E \to M$, there is a diffeomorphism (actually an isomorphism of principal bundles)
\[
 \Fr(E) \to \Fr(E^*), \quad \phi \mapsto \phi^*.
\]

\begin{prop}\label{prop_PB_groupoid_dual_VB_groupoid}
 Let $E_\cG \tto E_M$ be a VB-groupoid of rank $(l,k)$ over a Lie groupoid $\CG\soutar M$, and $E_\cG^* \tto C^*$ its dual VB-groupoid (Example \ref{ex_dual_VB_groupoid}). Then there is a canonical PB-groupoid isomorphism
  \[
 \Phi \colon \Fr^\sbis (E_\cG) \to \Fr^\tbis (E_\cG^*), \quad \phi_g \mapsto \phi_g^*,
 \]
over the diffeomorphism
\[
\Fr(C) \times_M \Fr(E_M) \to \Fr(C^*) \times_M \Fr(E_M^*), \quad (\varphi^c,\varphi^b) \mapsto ((\varphi^c)^*, (\varphi^b)^*),
\]
and with respect to the Lie 2-groupoid isomorphism. 
\[
 \GL(l,k) \to \GL(l,k), \quad A \mapsto A^T.
\]
\end{prop}

Note that $E_M^*$ is the core of $E_\CG^*$ and $C^*$ its base.
 
\begin{proof}
Given an element $\phi \in \Fr^\sbis (E_\cG)$, its corresponding dual frame $\phi^* $ belongs to $\Fr^\tbis (E_\cG^*)$ by Lemma \ref{lem.dual.s.bisec}. The fact that $\Phi$ is equivariant with respect to the right and the left principal 2-action of $\GL(l,k)$ and the isomorphism $A \mapsto A^T$ follows from
\[
 (\phi \cdot A)^* = A^T \cdot \phi^*. \qedhere
\]
\end{proof}

\begin{lem}\label{lem.dual.s.bisec}
	Let $E_\CG\soutar E_M$ be a VB-groupoid of rank $(l,k)$ over a Lie groupoid $\CG\soutar M$. Given any \frs\ frame $\phi\in \Fr^\sbis(E_\CG)$, its dual frame $\phi^* \in \Fr(E_\cG^*)$ is a \frt\ frame of the dual VB-groupoid $E_\CG^*\soutar C^*$, which has rank $(k,l)$.
\end{lem}

\begin{proof}
	Recall from Example \ref{ex_dual_VB_groupoid} that the source and target maps of $E^*_\cG \tto C^*$ are given by
	$$\tis^*(\xi)(v)=-\xi(0_g\circ \tii(v)) \y \tit^*(\xi)(v)=\xi(v\circ 0_g).$$
	Their kernels at $g\in \CG$ are
	$$\left\{\,(v\circ 0_g) \st v\in C_{\bs(g)\, }\right\}=\ker(\tis)_g \y \left\{\,(0_g\circ \tii(v)) \st v\in C_{\bt(g)} \,\right\}=\ker(\tit)_g,$$
	so that $\ker(\tit^*)=(\ker(\tis)^0)$ and $\ker(\tis^*)=(\ker(\tit)^0)$, where $0$ denotes the annihilator of those spaces. As a consequence, for any \frs\ frame $\phi$ of $E_\CG$, its dual $\phi^*$ satisfies the first condition in Definition \ref{def_special_frames_t} by the analogue of Lemma \ref{lemma_equivalent_def_frame_taking_values_in_kers}, since $\phi^*|_{\{0\}\times\KR^k}$ is a frame for $\ker(\tis)^0=\ker(\tit^*)$. Moreover, $\Ima(\phi^*|_{\KR^l\times \{0\}})$ is transverse to $\ker(\tit^*)$. We are left to show that $\phi^*|_{\KR^l\times \{0\}}$ is transverse to $\ker(\tis^*)=(\ker(\tit)^0)$; this is equivalent to the following linear map being injective
	$$\varphi\colon \ker(\tit)\fto \KR^l, \quad v\mapsto (\phi^*(e_1)(v),\dots,\phi^*(e_l)(v)).$$ 
	We notice first that any $v\in \ker(\tit)$ can be written as
	$$v=\sum_{i=1}^l\left(\phi^*(e_i)(v)\right)\phi(e_i)+\sum_{i=l+1}^{l+k}\left(\phi^*(e_i)(v)\right)\phi(e_i).$$
	If $v\in \ker(\varphi)$, then the first $l$ coordinates vanish, so $v\in {\rm Span}(\phi(e_{l+1}),\dots,\phi(e_{l+k}))$. Since $\phi$ is an \frs\ frame, the last $k$ coordinates must be transverse to $\ker(\tit)$, and therefore $v=0$. We conclude that $\phi^*$ satisfies the first condition in Definition \ref{def_special_frames_t}, hence $\phi^* \in \Fr(E^*_\cG)$.
\end{proof}

Composing the isomorphism $\Psi^{-1}$ from Proposition \ref{prop_isomorphism_s_t_bisection_frames} 
and $\Phi$ from Proposition \ref{prop_PB_groupoid_dual_VB_groupoid}
, one obtains the following conclusion.

\begin{cor}\label{cor_duality_PB}
  Let $E_\cG \tto E_M$ be a VB-groupoid over a Lie groupoid $\CG\soutar M$, and $E_\cG^* \tto C^*$ its dual VB-groupoid (Example \ref{ex_dual_VB_groupoid}). Then the PB-groupoid $\Fr^\sbis(E_\cG) \tto \Fr(C) \times_M \Fr(E_M)$ is canonically isomorphic to the PB-groupoid $\Fr^\sbis(E_\cG^*) \tto \Fr(E_M^*) \times_M \Fr(C^*)$. 
\end{cor}

\begin{ex}
Given a VB-groupoid $E_\cG \tto E_M$ with trivial core, its PB-groupoid $\Fr(E_\cG) \tto \Fr(E_M)$ (Example \ref{ex_rank_0k}) is canonically isomorphic to the PB-groupoid $\Fr(E^*_\cG) \tto \Fr(E_M^*)$ associated to the VB-groupoid $E_\cG^* \tto M \times \{0\}$ with trivial base and core $E^*_M$ (Example \ref{ex_rank_l0}).
\end{ex}

\addcontentsline{toc}{section}{References}

\printbibliography

\end{document}